\documentclass[reqno,10pt]{amsart}
\usepackage{amsmath,amsfonts,amssymb,amscd}
\usepackage[utf8]{inputenc}
\usepackage[matrix,arrow,ps]{xy}

\let\oldmarginpar\marginpar
\renewcommand\marginpar[1]{
\-\oldmarginpar[\raggedleft\footnotesize #1]%
{\raggedright\footnotesize #1}}
\theoremstyle{plain}
\newtheorem{theorem}[subsection]{Theorem}

\newtheorem{lemma}[subsection]{Lemma}
\newtheorem{proposition}[subsection]{Proposition}

\theoremstyle{definition}

\newtheorem{example}[subsection]{Example}
\theoremstyle{remark}
\newtheorem{remark}[subsection]{Remark}
\numberwithin{equation}{section}
\newcommand{\IA}{{\mathbb A}}
\newcommand{\IC}{{\mathbb C}}
\newcommand{\IF}{{\mathbb F}}

\newcommand{\IP}{{\mathbb P}}
\newcommand{\IQ}{{\mathbb Q}}
\newcommand{\IR}{{\mathbb R}}
\newcommand{\IZ}{{\mathbb Z}}
\newcommand{\ko}{\mathcal{O}}
\newcommand{\kv}{\mathcal{V}}
\newcommand{\kw}{\mathcal{W}}
\newcommand{\gothb}{\mathfrak{b}}
\newcommand{\gothf}{\mathfrak{f}}
\newcommand{\gothg}{\mathfrak{g}}
\newcommand{\gothh}{\mathfrak{h}}

\newcommand{\gothS}{\mathfrak{S}}
\newcommand{\kf}{\mathcal{F}}

\newcommand{\lra}{\longrightarrow}

\newcommand{\xra}{\xrightarrow}

\newcommand{\isom}{\cong}

\DeclareMathOperator{\Pic}{Pic}

\DeclareMathOperator{\Ker}{Ker}

\DeclareMathOperator{\id}{id}
\DeclareMathOperator{\ad}{ad}
\DeclareMathOperator{\res}{res}
\DeclareMathOperator{\LieSl}{SL}
\DeclareMathOperator{\Liesl}{sl}

\DeclareMathOperator{\LieGl}{Gl}

\DeclareMathOperator{\Lieso}{\mathfrak{so}}
\DeclareMathOperator{\LieSO}{SO}
\DeclareMathOperator{\Liesp}{\mathfrak{sp}}
\DeclareMathOperator{\LieSp}{Sp}

\DeclareMathOperator{\codim}{codim}
\DeclareMathOperator{\Sing}{Sing}	
\DeclareMathOperator{\Spec}{Spec}	
\DeclareMathOperator{\Ad}{Ad}

\DeclareMathOperator{\Amp}{Amp}
\DeclareMathOperator{\rk}{rk}
\DeclareMathOperator{\tr}{tr}

\DeclareMathOperator{\GIT}{/\!\!/}
\DeclareMathOperator{\pf}{pf}
\renewcommand{\SS}{S}
\newcommand{\So}{S_0}
\newcommand{\g}{\mathfrak{g}}
\newcommand{\h}{\mathfrak{h}}
\renewcommand{\b}{\mathfrak{b}}
\newcommand{\ie}{\emph{i.e.}\,}
\newcommand{\kO}{\mathcal{O}}

\newcommand{\NCone}{N}
\newcommand{\Wgp}{W}
\DeclareMathOperator{\tensor}{\otimes}
\DeclareMathOperator{\Blowup}{Bl}

\begin{document}
\title{Slodowy Slices and Universal Poisson Deformations}
\author{M.~Lehn, Y. Namikawa and Ch.~Sorger}

\address{Manfred Lehn\\
Fachbereich Physik, Mathematik u. Informatik\\
Johannes Gutenberg--Uni\-ver\-si\-tät Mainz\\
D-55099 Mainz, Germany}
\email{lehn@mathematik.uni-mainz.de}

\address{
Yoshinori Namikawa\\
Department of Mathematics\\
Faculty of Science\\
Kyoto University, Kitashirakawa-Oiwakecho, Kyoto, 606-8502, Japan}
\email{namikawa@math.kyoto-u.ac.jp}

\address{
Christoph Sorger\\
Laboratoire de Mathématiques Jean Leray (UMR 6629 du CNRS)\\
Université de Nantes\\
2, Rue de la Houssinière\\
BP 92208\\
F-44322 Nantes Cedex 03, France}
\email{christoph.sorger@univ-nantes.fr}


\begin{abstract}
We classify the nilpotent orbits in a simple Lie algebra for
which the restriction of the adjoint quotient map to a Slodowy slice is the
universal Poisson deformation of its central fibre. This generalises
work of Brieskorn and Slodowy on subregular orbits. In particular, we
find in this way new singular symplectic hypersurfaces of dimension 4
and 6.
\end{abstract}

\maketitle

\markboth{Slodowy Slices}{Lehn, Namikawa, Sorger}


\phantom{m}\hfill{\sl To the memory of Professor Masaki Maruyama}\\

\section{Introduction}\label{section:Introduction}

The purpose of this paper is twofold: Firstly, we shall explain how to generalise the
classical theorem of Grothendieck-Brieskorn-Slodowy on slices to the subregular nilpotent
orbit in a simple Lie algebra to arbitrary nilpotent orbits. The main idea here is to put  
the problem in the framework of Poisson deformations. Secondly, we shall describe new 
examples of singular symplectic hypersurfaces. These can be seen as higher dimensional 
analogues of the Kleinian or DuVal ADE-surface singularities. They arise as slices to very 
special nilpotent orbits. 

(1) Let $\g$ be a simple complex Lie algebra and consider the characteristic or quotient map
$\varphi:\g\to\g\GIT G$ for the action of the adjoint group $G$ of $\g$. The nullfibre 
$N=\varphi^{-1}(0)$ consists of the nilpotent elements in $\g$ and is called the nilpotent cone. 
It is an irreducible variety and decomposes into finitely many orbits. The dense orbit is called 
the regular orbit and denoted by $\ko_{reg}$, its complement $N\setminus \ko_{reg}$ is again
irreducible, its dense orbit is called the subregular orbit and denoted by $\ko_{sub}$.

A slice to the adjoint orbit of a nilpotent element $x\in N$ can be constructed as 
follows: The Jacobson-Morozov theorem allows one to find elements $h,y\in\g$ such that $x$, $h$
and $y$ form an $\Liesl_2$-triplet, i.e.\
satisfy the standard commutator relations $[h,x]=2x$, $[h,y]=-2y$ and $[x,y]=h$. The affine space
\begin{equation}
\SS=x+\Ker(\ad y).
\end{equation}
is a special transversal slice (\cite{Slodowy-LectureNotes}, 7.4) to the orbit through $x$.
We follow the tradition to refer to this choice as a \emph{Slodowy slice}.

The following theorem was conjectured by Grothendieck and proved by Brieskorn in his address
\cite{Brieskorn} to the International Congress in Nice 1974. In its original form it is stated
for groups. The following version for Lie algebras is taken from Slodowy's notes 
\cite[1.5 Theorem 1]{Slodowy-Utrecht}.

\begin{theorem} {\em (Grothendieck-Brieskorn)} \label{th:Brieskorn} --- 
Let $\g$ be a simple Lie algebra of type ADE, let $N$ be its nilpotent cone and let $S$ be a slice
to the orbit of a subregular nilpotent element in $\g$. Then the germ $(S\cap N,x)$ is a Kleinian
surface singularity with the same Coxeter-Dynkin diagram as $\g$, and the restriction 
$\varphi|_S:(S,x)\to (\g\GIT G,0)$ of the characteristic map $\varphi$ is isomorphic to the semi-universal 
deformation of the surface singularity $(S\cap N,x)$. 
\end{theorem}

An immediate obstacle to extending this theorem to other nilpotent orbits deeper down in the
orbit stratification of the nilpotent cone is the fact that the intersection $S_0=S\cap N$ is
no longer an isolated singularity so that there simply is no versal deformation theory. The
solution to this problem is to notice that $S_0$ carries a natural \emph{Poisson} structure,
that $\varphi|_S:S\to \g\GIT G$ can be considered as a deformation of Poisson varieties, and
that the space of infinitesimal Poisson deformation is again finite dimensional.

Recall that a Poisson structure on an $R$-algebra $A$ is an $R$-bilinear Lie bracket 
$\{-,-\}:A\times A\to A$ that satisfies $\{ab,c\}=a\{b,c\}+b\{a,c\}$ for all $a,b,c\in A$. 
A (relative) Poisson scheme is a morphism $f:X\to Y$ of finite type such that the structure
sheaf $\ko_X$ carries an $\ko_Y$-bilinear Poisson structure. 
A Poisson deformation of a Poisson variety $X$
over a pointed space $(T,t_0)$ is a flat morphism $p:{\mathcal X}\to T$ with a Poisson structure
on ${\mathcal X}$ relative over $T$ together with a Poisson isomorphism from $X$ to the fibre 
${\mathcal X}_{t_0}$. We will recall the basic properties of, and main results about, Poisson 
deformations relevant for this paper in section \ref{section:ReviewPoissonDeformations}.

Returning to the notations introduced above, it turns out that the restriction $\varphi_S:=
\varphi|_S:S\to \g\GIT G$ carries a natural relative Poisson structure for the Slodowy slice to
any nilpotent element. This Poisson structure is essentially induced by the Lie bracket on $\g$. 
It provides Poisson structures on each fibre of $\varphi_S$, and we may consider $\varphi_S$ as
a Poisson deformation of $S_0$ over the base $\g\GIT G$ (cf.\ the article of Gan and Ginzburg 
\cite{GanGinzburg} and section \ref{section:SlodowySlices}).

%
%

In order to state our first main theorem we need to introduce one more piece of notation: Let 
$\pi:\widetilde{\NCone}\rightarrow\NCone$ denote the Springer resolution of the nilpotent cone
(cf.\ section \ref{section:SimultaneousResolutions}). The so-called Springer fibre
$F_x:=\pi^{-1}(x)$ of $x\in N$ is the variety of all Borel subalgebras $\gothb\subset\g$ such
that $x\in\gothb$. Keeping the previous notations we can say:

\begin{theorem}\label{th:mainTheoremSimplyLaced} --- Let $x$ be a non-regular nilpotent element. 
Then $\varphi_\SS:\SS\to\g\GIT G$ is the formally universal Poisson deformation of $\So$ if and
only if the restriction map
$\rho_x:H^2(\widetilde{\NCone},\IQ)\rightarrow H^2(F_x,\IQ)$ is an isomorphism.
\end{theorem}

In the theorem, `formally universal' means that for any Poisson deformation $S' \to T$
of $\So$ over a local Artinian base $t_0\in T$, there is unique map $f:T\to\g\GIT G$ such that 
$f(t_0)=0$ and such that $S'$ is isomorphic, as a Poisson deformation of $\So$, to the 
pullback of $\SS$ under $f$. It is a subtle problem which conditions should be imposed
on the total space of the Poisson deformation of an affine Poisson variety when the base space is not
local Artinian. For the moment being, `formal universality' is the best we can hope for.

We will see that for any non-regular $x\in\NCone$ the map $\rho_x$ is injective (Proposition 
\ref{prop:Injectivity}). The following theorem clarifies the questions for which orbits the map 
$\rho_x$ is indeed an isomorphism: 

\begin{theorem}\label{th:mainTheoremNonSimplyLaced}--- Let $x$ be a non-regular nilpotent element.
Then the restriction map $\rho_x$ is an isomorphism except in the following cases:
\begin{itemize}
\item[($B_n$)] the subregular orbit,
\item[($C_n$)] orbits of Jordan types $[n,n]$ and $[2n-2i,2i]$, for $1\leq i\leq \tfrac{n}{2}$,
\item[($G_2$)] the orbits of dimension 8 and 10,
\item[($F_4$)] the subregular orbit.
\end{itemize}
In particular, $\rho_x$ is an isomorphism for all non-regular nilpotent elements in a simply laced
Lie algebra.
\end{theorem}

Consider the special case of a subregular nilpotent element $x$. If $\g$ is simply laced, $\So$ is
a surface with a corresponding ADE-singularity at $x$ then Theorems \ref{th:mainTheoremSimplyLaced} and \ref{th:mainTheoremNonSimplyLaced}
provide a Poisson version of Brieskorn's Theorem \ref{th:Brieskorn}. Note that even in this case our 
theorem claims something new: $\varphi_\SS$ is a \emph{formally universal} Poisson deformation, whereas
in the sense of usual flat deformations it is only semi-universal, which means that the classifying maps 
to the base $(\g\GIT G,0)$ are unique only on the level of tangent spaces (\cite{Slodowy-LectureNotes}, Section 2.3). 
If $\g$ is not simply laced, it follows from 
Slodowy's results \cite{Slodowy-LectureNotes} that $\varphi_\SS:\SS\to\g\GIT G$ cannot be
the universal Poisson deformation for subregular orbits.
%

The condition that $\rho_x$ be an isomorphism also appears as a hypothesis in recent work of 
Braverman, Maulik and Okounkov \cite{Braverman-Maulik-Okounkov}. They use $\rho_x$ in order 
to describe explicitly a quantum multiplication operator on the quantum cohomology of the inverse
image of the Slodowy slice $\So$ in $\NCone$ under the Springer resolution. They show, with different 
arguments, that $\rho$ is an isomorphism for simply laced $\g$ (\cite{Braverman-Maulik-Okounkov}, 
Appendix). Thus our Theorem \ref{th:mainTheoremNonSimplyLaced} extends the range of cases where 
their Theorem 1.3 \cite{Braverman-Maulik-Okounkov} applies.

In the exceptional cases of Theorem \ref{th:mainTheoremNonSimplyLaced}, the Slodowy slices do not give
the universal Poisson deformations of $\So$ and it is natural to ask what the universal Poisson deformations
are. We will restrict ourselves to the Poisson deformation not of the affine variety $S_0$, but of the germ 
$(\So,x)$ in the complex analytic category. 

If $x$ is a subregular nilpotent element in a Lie algebra of type $B_n$, $C_n$, $F_4$ and $G_2$, then 
$(\So,x)$ is a surface singularity of type $A_{2n-1}$, $D_{n+1}$, $E_6$ and $D_4$, respectively. One can therefore
construct its universal Poisson deformation as a Slodowy slice in the corresponding simply laced Lie algebra. 

If $x$ belongs to the $8$-dimensional or `subsubregular' orbit in the Lie algebra of type $G_2$, then $(\So,x)$
turns out to be isomorphic to the analogous singularity $(\So',x')$ for a nilpotent element $x'\in \Liesp_{6}$ of 
Jordan-type $[4,1,1]$. As this orbit is not in the exceptional list of Theorem \ref{th:mainTheoremNonSimplyLaced} 
the associated Slodowy slice provides the universal Poisson deformation both of $(\So',x')$ and $(\So,x)$.   

In section \ref{section:DualPairs} we shall discuss an analogous phenomenon for the remaining cases, 
namely the orbits in $\Liesp_{2n}$ of Jordan types $[n,n]$ and $[2n-2i,2i]$. Kraft and Procesi \cite{K-P} already 
observed this phenomenon without the Poisson point of view. We will clarify that their method is closely related
to Weinstein's notion of a {\em dual pair} in Poisson geometry (\cite{W}).   

(2) One initial motivation for this article was the search for singular symplectic hypersurfaces.
A symplectic variety is a normal variety $X$ with a closed non-degenerate 2-form $\omega$ on its
regular part that extends as a regular 2-form to some (and then any) proper resolution $f:X'\to X$
of the singularities of $X$. Symplectic varieties carry natural Poisson structures: On the regular
part, the form $\omega:T_{X_{reg}}\to \Omega_{X_{reg}}$ can be inverted to yield a map
$\theta:\Omega_{X_{reg}}\to T_{X_{reg}}$. One checks that the bracket $\{f,g\}=\theta(df)(dg)$ on 
$\ko_{X_{reg}}$ satisfies the Jacobi-identity since $\omega$ is closed. By normality, this 
Poisson structure canonically extends to $X$. Many examples of singular symplectic varieties 
that we are aware of (like nilpotent orbit closures, finite group quotients, symplectic reductions) indicate
that symplectic singularities tend to require large embedding codimensions. In particular, 
singular symplectic hypersurfaces should be rare phenomena. Previously known were only the 
Klein-DuVal surface singularities in $\IC^3$. We found the following new examples.

Firstly, there is a series of four-dimensional symplectic hypersurfaces that appear as intersections
$S_0=N\cap S$ of the nilpotent cone $N$ with Slodowy slices $S$ to certain nilpotent orbits in 
$\Liesp_{2n}$. In simplified coordinates these can be written as follows:

\begin{example}\label{ex:hypersurfaces1} --- For each $n\geq 2$ the following polynomial defines
a four-di\-men\-sio\-nal symplectic hypersurface:
$$f=a^2x+2aby+b^2z+(xz-y^2)^n\in\IC[a,b,x,y,z].$$ 
\end{example}

Secondly, we have a single six-dimensional example that appears in a similar way in the exceptional
Lie algebra $\gothg_2$. The corresponding polynomial $f$ in seven variables can be best expressed
in the following way: Consider the standard action of the symmetric group $\gothS_3$ on $\IC^2$ and
the corresponding symplectic action on $\IC^2\oplus (\IC^2)^*$. The invariant ring of the latter action
is spanned by
seven elements, say $a,b,c$ of degree 2 and $p,q,r,s$ of degree 3, that are obtained by polarising
the second and the third elementary symmetric polynomial in 3 variables. The ideal of the quotient
variety $(\IC^2\oplus \IC^{2*})/\gothS_3\subset \IC^7$ is generated by the following five relations
among the invariants: 
$$
\begin{array}{l}
t_1=a(ac-b^2)+2(q^2-rp)\\
t_2=b(ac-b^2)+(rq-ps)\\
t_3=c(ac-b^2)+2(r^2-qs)
\end{array}\quad\quad
\begin{array}{l}
z_1=as-2br+cq\\z_2=ar-2bq+cp
\end{array}
$$
Keeping this notation we can say:
\begin{example}\label{ex:hypersurfaces2} --- The following polynomial defines a six-dimensional
symplectic hypersurface:
$$f=z_1^2a-2z_1z_2b+z_2^2c+2(t_2^2-t_1t_3)\in\IC[a,b,c,p,q,r,s].$$
\end{example}

In both cases the Poisson structure and in turn the symplectic structure can be recovered from a 
minimal resolution of the Jacobian ideal that is generated by the partial derivatives of $f$.
We will return to such issues in a later article.

\section{Poisson deformations}\label{section:ReviewPoissonDeformations}

For the convenience of the reader, we shall briefly review in this section some aspects of the
theory of Poisson deformations. For
details and further information we refer to the articles of Ginzburg and Kaledin \cite{GinzburgKaledin} 
and the second author \cite{Namikawa-Flops,Namikawa-PoissonDeformations1,Namikawa-PoissonDeformations2}.

Let $(X, \{\;, \;\})$ be an algebraic variety with a Poisson structure. We will usually denote the
pair again by the simple letter $X$ and suppress the bracket if no ambiguity can arise.
Let $A$ be a local Artinian $\IC$-algebra with residue filed $A/m_A = \IC$ and let $T=\Spec(A)$.
A Poisson deformation of $X$ over $A$ is a flat morphism $\mathcal{X}\to T$ with a relative Poisson 
structure $\{\;, \;\}_T$ on ${\mathcal X}/T$ and an isomorphism
$\phi:X\to\mathcal{X}\times_T\mathrm{Spec}(\IC)$ of Poisson varieties.

We define $\mathrm{PD}_X(A)$ to be the set of equivalence classes of such pairs $(\mathcal{X}/T,\phi)$ 
where $(\mathcal{X}, \phi)$ and $(\mathcal{X}', \phi')$ are defined to be equivalent if there is a
Poisson isomorphism $\psi: \mathcal{X} \cong \mathcal{X}'$ over $T$ with $\psi\circ\phi=\phi'$.
We obtain in this way the {\em Poisson deformation functor}:  
$$\mathrm{PD}_{X}: (\mathrm{Art})_{\IC} \to (\mathrm{Set})$$ 
from the category of local Artin $\IC$-algebras with residue field $\IC$ to the category of sets.  
Let $\IC[\epsilon]$ be the ring of dual numbers. The set $\mathrm{PD}_X(\IC[\epsilon])$ has the
structure of a $\IC$-vector space and is called the tangent space of $\mathrm{PD}_X$. A Poisson 
deformation of $X$ over $\mathrm{Spec}{\IC}[\epsilon]$ is called a {\em 1-st order} Poisson
deformation of $X$. 

As a particularly interesting case, consider an affine symplectic variety $X$ with a symplectic structure $\omega$ 
(cf. Introduction (2)). 
Assume further that there exists a symplectic projective resolution $\pi:Y\to X$, i.e.\ a projective 
resolution with the property that $\omega$ extends to a \emph{symplectic} form on $Y$. This is equivalent
to requiring that $\pi$ be crepant. (One can replace $Y$ by a $\IQ$-factorial terminalisation of $X$ if 
$X$ does not have a crepant resolution.) As explained in the introduction, both $X$ and $Y$ carry natural Poisson structures. 
Moreover, if $p:{\mathcal Y}\to T$ is a Poisson deformation of $Y$, one can show that 
$\mathcal X:=\Spec(p_*\ko_{\mathcal Y})\to T$ is a Poisson deformation of $X$. This defines a natural map of functors 
$$ \pi_*:\mathrm{PD}_Y \to \mathrm{PD}_X.$$ 
Finally, assume that $X$ has a $\IC^*$ action with positive weights such that $\omega$ becomes
homogeneous of positive weight. (In particular, $X$ is contractible.) In this case, the $\IC^*$
action uniquely extends to $Y$. 

Under these assumptions and with the introduced notation one has the following theorem that combines
results from \cite{Namikawa-PoissonDeformations1} and \cite{Namikawa-PoissonDeformations2}.

\begin{theorem} {\em (Namikawa)} \label{th:NamikawaDeformation} --- The Poisson deformation functors $\mathrm{PD}_X$ and
$\mathrm{PD}_Y$ are prorepresentable and unobstructed. More precisely, there is a $\IC^*$-equivariant
commutative diagram 
\begin{equation} 
\begin{CD} 
\mathcal{Y} @>>> \mathcal{X}  \\ 
@VVV @VVV \\ 
\mathbb{A}^d @>{\psi}>> \mathbb{A}^d    
\end{CD} 
\end{equation} 
where $\IA^d$ is the affine space of dimension $d=\dim H^2(Y,\IC)$ isomorphic to $\mathrm{PD}_X(\IC[\epsilon])$
resp.\ $\mathrm{PD}_Y(\IC[\epsilon])$, such that $\mathcal{X} \to \mathbb{A}^d$ and $\mathcal{Y} \to \mathbb{A}^d$ 
are formally universal Poisson deformations of $X$ and $Y$, resp., at $0\in\mathbb{A}^d$,
and the map $\psi$ is compatible with the functor map $\pi_*:\mathrm{PD}_Y \to \mathrm{PD}_X$. 
Moreover, $\psi$ is a finite Galois cover with $\psi (0) = 0$. 
\end{theorem}

\section{Slodowy slices}\label{section:SlodowySlices}

Let $\g$ be a simple complex Lie algebra and $G$ its adjoint group. Let $x\in\g$ be a nilpotent
element. According to the Jacobson-Morozov theorem, there is a nilpotent element $y\in\g$ and a
semi-simple element $h\in\g$ such that $[h,x]=2x$, $[h,y]=-2y$ and $[x,y]=h$. The resulting triple
$\{x,h,y\}$, called a Jacobson-Morozov triple for $x$, defines a Lie algebra homomorphism
$\Liesl_2\to\g$ that is non-zero and hence an embedding if $x\neq 0$.
Slodowy (\cite{Slodowy-LectureNotes}, 7.4) showed that the affine space
$$\SS=x+\Ker(\ad y)$$
is a transverse slice to the conjugacy class of $x$. One obtains a natural $\IC^*$-action on $\g$
as follows: The $\ad h$-action yields a decomposition $\g=\oplus_{d\in\IZ}\g(d)$ into 
weight spaces $\g(d)=\{z\in\g\ |\ [h,z]= d\,z\}$. Define $\rho_t(z)=t^{2-d}z$ for $z\in \g(d)$
and extend linearly. This action fixes the nilpotent element $x$ and stabilises the slice $S$. 
%
%
Nilpotent orbits that intersect $\SS$ do so transversely. This is clear at $x$, hence in some open 
neighbourhood of $x$ in $\SS$, then everywhere as $\IC^*$ acts with positive weights.

Each fibre of $\varphi:\g\to\g\GIT G$ contains a unique conjugacy class of a semi-simple element.
In this sense, $\varphi$ maps an element $z\in\g$ to the class $[z_s]$ of its semi-simple part. In 
particular, its central fibre is the nilpotent cone $\NCone\subset\g$. Choose once and for all a Cartan 
subalgebra $\h\subset\g$ containing $h$. Let $\Wgp$ be the associated Weyl group. According to Chevalley, 
the inclusion $\h\subset\g$ induces an isomorphism $\IC[\g]^G\to\IC[\h]^\Wgp$, hence an identification
$\h/\Wgp\isom\g\GIT G$. The morphism $\varphi:\g\to\h/\Wgp$ is called the \emph{adjoint quotient}.

According to classic results of Kostant \cite{Kostant} extended by Slodowy (\cite{Slodowy-LectureNotes}, 
5.2) and Premet (\cite{Premet}, 5), its restriction $\varphi_\SS:\SS\to\h/\Wgp$ to the Slodowy slice is
faithfully flat (hence surjective) with irreducible, normal complete intersection fibres of dimension
$\dim\SS-\rk\g$ whose smooth points are exactly the regular elements of $\g$ contained in the fibre. In 
particular, the central fibre $\So=\SS\cap\NCone$ is an irreducible normal complete intersection
whose regular points are the regular nilpotent elements in $\SS$.

The Lie bracket on $\g$ extends uniquely to a Poisson structure on the symmetric algebra $S\g=\IC[\g^*]$. 
By construction, the invariant subalgebra $(S\g)^G$ Poisson commutes with all elements in $S\g$. 
Identifying $\g$ with $\g^*$ via the Killing form defines a Poisson structure on $\g$ relative to 
$\h/\Wgp$.

According to Gan and Ginzburg  the Slodowy slice $\SS$ inherits a $\IC^*$-invariant Poisson structure of 
weight $-2$ via a Hamiltonian reduction (\cite{GanGinzburg}, 3.2). This description also shows that on 
regular points of the fibres, this Poisson structure corresponds to the Kostant-Kirillov-Souriau 
symplectic form (see also \cite{Yamada}).

Thus, $\varphi_\SS:\SS\to\h/\Wgp$ is a Poisson deformation of the central fibre $\So$.

\section{$\IQ$-factoriality of the nilpotent cone}

The nilpotent cone decomposes into the disjoint union of finitely many nilpotent orbits. The
dense orbit of maximal dimension corresponding to regular nilpotent elements in $\g$ is called
the regular orbit $\kO_{reg}$; its complement $\NCone\setminus\kO_{reg}$ also contains a dense
orbit, corresponding to subregular nilpotent elements in $\g$, the so-called subregular orbit
$\kO_{sub}$. If $\SS$ is a Slodowy slice to an element in $\kO_{sub}$, then $\So=\SS\cap\NCone$
is a surface singularity of ADE-type, and as such a quotient $\IC^2/\Gamma(\g)$ for a finite
subgroup $\Gamma(\g)\subset\LieSl_2$.

\begin{proposition}\label{prop:LocallyAnalyticallyQFactorial} ---
Let $m=|\Gamma(\g)|$. The nilpotent cone is analytically locally $m$-factorial.
\end{proposition}

Recall that a normal variety $Z$ is locally \emph{$m$-factorial} if for any local Weil divisor
$D$ of $Z$, the divisor $mD$ is Cartier.

\begin{proof} Recall the following general facts. Let $X$ be a reduced complex analytic space
and consider a Whitney stratification on $X$. Denote by $X_i$ the union of the strata of
dimension $\leq i$. The \emph{rectified homological depth} of $X$ is said to be $\geq n$ if for
any point $x\in X_i\setminus X_{i-1}$, there is a fundamental system $(U_\alpha)$ of neighbourhoods
of $x$ in $X$ such that, for any $\alpha$ we have $H_k(U_\alpha,U_\alpha\setminus X_i;\IZ)=0$
for $k<n-i$. The vanishing of homology still holds if we replace $U_\alpha$ by a sufficiently small 
contractible open neighbourhood $U$ of $x$. The basic result we will use is that if $X$ is locally
a non empty complete intersection, then the rectified homological depth of $X$ is equal to the
dimension of $X$ (see \cite{HammLeDT}, Corollary 3.2.2 and use the relative Hurewicz theorem).

Recall that $\NCone$ is a complete intersection. We will apply the above for the stratification
given by the orbits. Our proof is by induction on the codimension $d(O)$ of the orbit inside 
$\NCone$. Note that $d(O)$ is even. When $d(O)=0$, every point $z\in O$ is actually a smooth 
point of $\NCone$, so the claim holds trivially. When $d(O)=2$, then $\NCone$ has ADE singularities
along $O$, hence $\NCone$ is $m$-factorial at $x\in O$ as a quotient $\IC^2/\Gamma(\g)$.
In the general case, consider a Weil divisor $D$ in an open contractible Stein neighbourhood $U$ of 
$x\in\NCone$. Let $B$ be the intersection of the orbit of $x$ with $U$. Note that $B\subset\Sing(U)$
but that $B$ itself is smooth of codimension $\geq 4$ in $U$.

By induction, for $w\in\Sing(U)\setminus B$, we know that $mD$ is Cartier in an open neighbourhood of 
$w\in\NCone$. Hence we may assume that $mD$ is Cartier in $U\setminus B$, \ie that there is a line
bundle $L$ on $U\setminus B$ and a section $s$ in $L$ such that $mD$ is the zero locus of $s$.
Consider the exact sequence of the pair $(U,U\setminus B)$
\begin{equation}\label{eq:CohomologyOfThePair}
H^2(U,\IZ)\to H^2(U\setminus B,\IZ)
\to H^3(U,U\setminus B;\IZ).
\end{equation}

Remark that $H^2(U,\IZ)=0$ as $U$ is Stein and contractible. Moreover we have $H_3(U,U\setminus B;\IZ)=0$ 
as the rectified homological depth of $\NCone$ is equal to the dimension of $\NCone$ and $B$ is of 
codimension $\geq 4$. By the universal coefficient theorem, it follows that $H^3(U,U\setminus B;\IZ)=0$ 
because $H_3(U,U\setminus B;\IZ)=0$.

From the exact sequence \eqref{eq:CohomologyOfThePair}, we see that  $H^2(U\setminus B,\IZ)=0$.
Now since $U$ is Cohen-Macaulay and $\codim_{U}B\geq 4$, we know that $H^i(U\setminus B,\kO)=0$ 
for $i\leq 2$. Using the exponential sequence, we see that $H^1(U\setminus B,\kO^*)\isom 
H^2(U\setminus B,\IZ)=0$ so that we have $\Pic^{an}(U\setminus B)=0$. It follows that $L$ is trivial
on $U\setminus B$. The section $s$ extends by normality. Hence $mD$ is Cartier on $U$.
\end{proof}

In the algebraic case, the result follows from parafactoriality of complete intersections.

\begin{remark} --- The nilpotent cone is algebraically locally $m$-factorial.
\end{remark}

\begin{proof} The assertion is clear on  the open subset $U=\kO_{reg}\cup\kO_{sub}$. Let $V\supset U$
be an open subset of $\NCone$ such that $V$ is locally $m$-factorial.

Let $\eta$ be a generic point of $\NCone\setminus V$ and let $D$ be a Weil divisor on 
$T=\Spec(\kO_{N,\eta})$. By the choice of $V$, there is a line bundle $L$ on $T\setminus\{\eta\}$
and a section $s$ in $L$ such that $mD$ is the zero locus of $s$.

Now the local ring $\kO_{N,\eta}$ is a complete intersection of dimension $\geq 4$ and hence 
parafactorial according to Grothendieck (\cite{Grothendieck}, Théorème 3.13 (ii)). This implies 
that $\Pic(T\setminus\{\eta\})=0$. In particular, $L$ is trivial and the section $s$ extends to a
section in $\Gamma(T,\kO_T)$ as $T$ is normal. This shows that $mD$ is Cartier on $T$.

The assertion follows by noetherian induction.
\end{proof}

\section{Simultaneous resolutions} \label{section:SimultaneousResolutions}

Let $\g$ be a simple complex Lie algebra and let $\h\subset\g$ be a Cartan subalgebra
as in section \ref{section:SlodowySlices}. Let $G$ be the adjoint group and $T\subset G$
the maximal torus corresponding to $\h$. For any Borel subgroup $B$ with
$T\subset B\subset G$ and with Lie algebra $\b$ consider the commutative diagram
\begin{equation}\label{eq:AdjointQuotientCD}
\begin{CD}
G\times^{B}\b @>{\pi_B}>> \g \\
@V\varphi_B VV @VV\varphi V \\
\h @>\pi >> \h/\Wgp
\end{CD}
\end{equation}
where $\pi$ is the quotient map and $\pi_B$ and $\varphi_B$
are defined as follows. For a class $[g,b]\in G\times^{B}\b$ let
$b=b_0+b_1$ be the decomposition corresponding to
$\b=\h\oplus[\b,\b]$.
Then  $\varphi_B([g,b])=b_0$
and $\pi_B([g,b])=\Ad(g)b$.
According to Grothendieck, the above diagram is a simultaneous resolution for $\varphi$. The central 
fibre of $\varphi_B$ is the cotangent bundle $T^*(G/B)$ of the flag variety $G/B$. The map
\begin{equation}
\pi_{B,0}: T^*(G/B)\to\NCone
\end{equation}
is a crepant resolution. The exceptional locus $E$ is a normal crossing divisor whose components 
correspond bijectively to the simple roots defined by the choice of $\b$ (\cite{Slodowy-Utrecht}, 3.2).

Let $z\in\NCone$ be a non-regular nilpotent element. Choose an open neighbourhood $V\subset T^*(G/B)$
of the fibre $F_z=\pi_{B,0}^{-1}(z)$ such that the inclusion $F_z\subset V$ is a homotopy equivalence.
As $\pi_{B,0}: T^*(G/B)\to\NCone$ is proper we may choose a contractible open Stein neighbourhood 
$U\subset\NCone$ of $z$ such that $\widetilde{U}:=\pi_{B,0}^{-1}(U)\subset V$.

\begin{lemma}\label{lemma:Pican} --- The map $\Pic^{an}(\widetilde{U})\xra{c_1}
H^2(\widetilde{U};\IZ)$ is an isomorphism. Moreover, $\Pic^{an}(\widetilde{U})_\IQ$
has a basis given by the line bundles ${\mathcal{O}}(E_i)$, where the $E_i$ are the
components of $\widetilde{U}\cap E$.
\end{lemma}

\begin{proof} We know that $\widetilde{U}\to U$ is a resolution of rational singularities. As $U$ is 
Stein, it follows from Leray's spectral sequence $H^p(U,R^q\pi_{_{B,0}*}\ko_{\widetilde{U}})
\Rightarrow H^{p+q}(\widetilde{U},\ko_{\widetilde{U}})$ that the cohomology groups 
$H^i(\widetilde{U},\ko)$ vanish for $i=1,2$. Hence the exponential  cohomology sequence
\begin{equation}
H^1(\widetilde{U},\ko)\to H^1(\widetilde{U},\ko^*)\to
H^2(\widetilde{U},\IZ)\to H^2(\widetilde{U},\ko)
\end{equation}
yields the first assertion.

Let $\widetilde{D}\subset\widetilde{U}$ be a prime divisor distinct from any $E_i$.
Then $D=\pi_{B,0}(\widetilde{D})$ is a Weil divisor on $U$. As $U$ is analytically
$\IQ$-factorial by Proposition \ref{prop:LocallyAnalyticallyQFactorial}, there is an integer $m$ 
such that $mD$ is Cartier. As $U$ is Stein and contractible, $\Pic^{an}(U)=0$ by a result of Grauert,
hence there is an analytic function $s\in\ko(U)$ such that $mD=\{s=0\}$. It follows that  
$\{\pi_{B,0}^*(s)=0\}=m\widetilde{D}+\sum m_i E_i$ for suitable coefficients. This shows that
$\Pic^{an}(\widetilde{U})_\IQ$ is generated by the line bundles ${\mathcal{O}}(E_i)$.

It remains to see that these line bundles are linearly independent. Note that every $E_i$ maps 
surjectively to a component of the singular locus of $U$ as $\pi_{B,0}: T^*(G/B)\to\NCone$ is 
semismall. Choose
a surface $A\subset U$ that intersects every component of the singular locus non trivially and in
such a way that $\widetilde{A}=\pi_{B,0}^{-1}(A)$ is smooth (Bertini). Let $C_j$ be the components
of the exceptional fibres of $\widetilde{A}\to A$. By a result of Grauert, the intersection matrix
of the curves $C_j$ is negative definite. In particular, the classes of the line bundles
$\ko_{\widetilde{A}}(C_j)$ are linearly independent. A fortiori, the line bundles $\ko(E_i)$ are
linearly independent.
\end{proof}

Keeping the notation we have furthermore:

\begin{lemma}\label{lemma:RestrictionIsInjection} --- Let $z^\prime\in U$ be a non-regular nilpotent 
element. Then the composition
\begin{equation}
H^2(F_z,\IQ)\xra{\res^{-1}} H^2(V,\IQ) \xra{\res} H^2(F_{z^\prime},\IQ)
\end{equation}
is injective. In particular, if $v\in\NCone$ is an arbitrary non-regular nilpotent element such that
$z\in\overline{G.v}$, then $\dim H^2(F_z,\IQ)\leq \dim H^2(F_{v},\IQ)$.
\end{lemma}

\begin{proof} As before, we may choose an open neighbourhood $V^\prime\subset\widetilde{U}$ of the 
fibre $F_{z^\prime}=\pi_{B,0}^{-1}(z^\prime)$ such that the inclusion $F_{z^\prime}\subset V^\prime$
is a homotopy equivalence and an open contractible Stein neighbourhood $U^\prime\subset U$ of
$z^\prime$ such that $\widetilde{U}^\prime:=\pi_{B,0}^{-1}(U^\prime)\subset V^\prime$.
We get the following diagram of restriction maps
\begin{equation}
\xymatrix{
 H^2(V,\IQ)\ar^a[r]\ar_i[d]& H^2(V^\prime,\IQ)\ar^{i^\prime}[d]\\
 H^2(\widetilde{U},\IQ)\ar^b[r]\ar_p[d]& H^2(\widetilde{U}^\prime,\IQ)\ar^{p^\prime}[d]\\
 H^2(F_z,\IQ)& H^2(F_{z^\prime},\IQ)
}
\end{equation}
Since $pi$ and $p^\prime i^\prime$ are isomorphisms, $i$ and $i^\prime$
are injective and $p$ and $p^\prime$ are surjective. Next,
$b$ is injective: to see this it suffices by Lemma \ref{lemma:Pican}
that every component $E_i$ of $\widetilde{U}\cap E$ intersects
with $\widetilde{U}^\prime$ non-trivially. But this is clear as
$E_i$ maps surjectively to the singular locus of $U$, hence hits $z^\prime$.
Since $i$ and $b$ are injective, it now follows that $a$ is injective.
Finally, $p^\prime i^\prime a$ is injective as claimed.
The assertion on the dimension follows of course from the fact that
the orbit of $v$ intersects $U$ and
that points in the same nilpotent orbits have isomorphic fibres.
\end{proof}

\begin{proposition}\label{prop:Injectivity} --- Let $z\in\NCone$ be a non-regular nilpotent element. 
The restriction map $H^2(T^*(G/B),\IQ)\to H^2(F_z,\IQ)$ is injective. In particular, it is an 
isomorphism if and only if $\dim H^2(F_z,\IQ) =\rk\g$.
\end{proposition}

\begin{proof} Choose a subregular nilpotent element $p\in U$. Then we have the
following diagram of restriction maps:
\begin{equation}\label{eq:Dreiecksdiagram}
\begin{array}{c}
\xymatrix{
    H^2(F_z,\IQ)& H^2(V,\IQ)\ar_{\res_z}[l]\ar^{\res_p}[r] & H^2(F_p, \IQ)\\
    & H^2(T^*(G/B),\IQ)\ar^{a_z}[lu]\ar^{a}[u]\ar_{a_p}[ru]
}
\end{array}
\end{equation}
The map $\res_z$ is an isomorphism by the choice of $V$. 
By (\cite{Slodowy-Utrecht}, 4.5 Corollary), the map $a_p$ is injective, 
hence $a_z$ is injective as well.
\end{proof}

Let $\SS$ be a Slodowy slice to an element $x\in\NCone$ as in section \ref{section:SlodowySlices}.
Define $\SS_B := \pi_B^{-1}(\SS)$. The induced commutative diagram
\begin{equation}\label{eq:SlodowySliceCD}
\begin{CD}
\SS_B @>\pi_{_{BS}}>> \SS \\
@V\varphi_{_{BS}} VV @VV\varphi_S V \\
\h @>\pi>> \h/\Wgp
\end{CD}
\end{equation}
is a simultaneous resolution of $\SS\to\h/\Wgp$ according to Slodowy (\cite{Slodowy-LectureNotes}, 5.3).
It yields a projective crepant resolution $\alpha_B:\SS_B\to\SS\times_{\h/\Wgp}\h$.

\begin{theorem}\label{th:CrepantResolutions} --- Assume that $x$ satisfies the following condition:
\begin{equation*}
(*)\quad\quad\dim H^2(F_x,\IQ)=\rk\g.
\end{equation*}
Then every projective crepant resolution $\alpha^\prime:\SS^\prime\to\SS\times_{\h/\Wgp}\h$
is isomorphic to $\alpha_B$ for some Borel subgroup $B$ with $T\subset B$.
\end{theorem}

We will prove the theorem in several steps.

\begin{proposition}\label{prop:Restriction} --- Assume condition $(*)$ holds.
Then the natural restriction map
\begin{equation}\label{eq:RestrictionToSliceOnH2Level}
H^2(G\times^{B}\b,\IQ)\lra  H^2(\SS_B,\IQ)
\end{equation}
is an isomorphism.
\end{proposition}

\begin{proof}  We denote by $\SS_{B,0}$ the central fibre of $\varphi_{_{BS}}$. Consider the 
commutative diagram
\begin{equation}
\xymatrix@!0{
 & & G\times^{B}\b \ar^(.65){\pi_{B}}[rr]
   & & \g
\\
 & T^*(G/B) \ar@{^{(}->}[ur]\ar[rr]
 & & N \ar@{^{(}->}[ur]
\\
 F_x\ar@{^{(}->}[dr]\ar@{^{(}->}[ur] & & \SS_B \ar'[r][rr]\ar'[u][uu]
   & & \SS\ar[uu]
\\
 &\SS_{B,0}\ar@{->}[uu] \ar[rr]\ar@{^{(}->}[ur]
 & & \SS_0\ar[uu]\ar@{^{(}->}[ur]
}
\end{equation}
Firstly, $H^2(G\times^{B}\b,\IQ)\to H^2(T^*(G/B),\IQ)$ is an isomorphism. Indeed, as
we know that $T^*(G/B)=G\times^{B}[\b,\b]$, this is clear from the inclusion $[\b,\b]\subset\b$.
Secondly, $H^2(\SS_B,\IQ)\to H^2(\SS_{B,0},\IQ)$ is an isomorphism. This follows from the
fact that $\varphi_{_{BS}}:\SS_B\to\h$ is a trivial fibre bundle of $C^\infty$-manifolds
(\cite{Slodowy-Utrecht}, Remark at the end of section 4.2). It follows that
\eqref{eq:RestrictionToSliceOnH2Level} is an isomorphism if and only if the restriction
\begin{equation}\label{eq:RestrictionToSliceCentralFiber}
H^2(T^*(G/B),\IQ)\to H^2(\SS_{B,0},\IQ)
\end{equation}
is. Now $F_x$ is the fibre over the unique $\IC^*$ fixed point $x$ of the proper equivariant
map $\SS_{B,0}\to\SS_0$. Hence $F_x\to \SS_{B,0}$ is a homotopy equivalence and the restriction
\begin{equation}\label{eq:RestrictionToF}
H^2(\SS_{B,0},\IQ)\xra{\sim}H^2(F_x,\IQ)
\end{equation}
is an isomorphism. The proposition now follows, using proposition \ref{prop:Injectivity},
from condition $(*)$.
\end{proof}

Let $\overline{\Amp}(\alpha_B)\subset H^2(G\times^{B}\b)$ be the relative nef cone of 
$\alpha_B:G\times^{B}\b\to\g \times_{\h/W}\h$ and similarly let 
$\overline{\Amp}(\alpha_{_{BS}})\subset H^2(\SS_B,\IR)$ be the nef cone of
$\alpha_{_{BS}}:\SS_B\to\SS\times_{\h/\Wgp}\h$.

\begin{proposition}\label{prop:NefCones} --- Assume that condition $(*)$ holds. The restriction 
isomorphism $H^2(G\times^{B}\b,\IR)\lra  H^2(\SS_B,\IR)$ maps $\overline{\Amp}(\alpha_B)$ onto
$\overline{\Amp}(\alpha_{_{BS}})$.
\end{proposition}

\begin{proof} As any relatively ample line bundle $L$ on $G\times^{B}\b$ restricts to a relatively
ample line bundle on $\SS_B$, it is clear that the image of $\overline{\Amp}(\alpha_B)$ is contained
in $\overline{\Amp}(\alpha_{_{BS}})$. The nef cone $\overline{\Amp}(\alpha_B)$ is a simplicial rational
polyhedral cone in $H^2(G\times^{B}\b,\IR)$ spanned by line bundles $L_i=G\times^{B}\IC(\chi_i^{-1})$
where $\IC(\chi_i)$ is the one dimension representation of $B$ corresponding to the character $\chi_i$
and $\chi_1,\dots,\chi_{\rk\g}$ are the fundamental weights of $B$. If $L$ is a general element in the
one codimensional face of $\overline{\Amp}(\alpha_B)$ opposite to $L_i$, then the morphism
$\varphi_{|mL|}$ associated to $L$ contracts for every positive $m$ at least one of the rational curves
in any fibre $F_p$ for $p$ subregular. In particular, the restriction of $L$ to $\SS_B$ cannot be ample 
and hence is contained in the boundary of $\overline{\Amp}(\alpha_{_BS})$. This shows that 
$\partial\overline{\Amp}(\alpha_B)$ is mapped to $\partial\overline{\Amp}(\alpha_{_BS})$.
\end{proof}

\begin{proof}[Proof of Theorem \ref{th:CrepantResolutions}.]
Recall that under the isomorphism
\begin{equation}
H^2(G\times^B\b,\IR)\isom H^2(G/T,\IR)\isom\h_{\IR}^*
\end{equation}
the nef cone $\overline{\Amp}(\alpha_B)$ is mapped to
the Weyl chamber determined by the choice of $B$. As
$\h_{\IR}^*$ is the union of all Weyl chambers, it follows that
\begin{equation}
H^2(G\times^B\b,\IR)=\bigcup_{w\in\Wgp}\overline{\Amp}(\alpha_{_{w(B)}})
\end{equation}
where $w(B)=wBw^{-1}$.
Proposition \ref{prop:NefCones} implies that
\begin{equation}\label{eq:NefCones}
H^2(\SS_B,\IR)=\bigcup_{w\in\Wgp}\overline{\Amp}(\alpha_{_{w(B)S}})
\end{equation}
Now recall the following general fact. Suppose that
$\beta^\prime:Z^\prime\to Z$ and
 $\beta^{\prime\prime}:Z^{\prime\prime}\to Z$
are projective crepant resolutions of an affine normal Gorenstein variety $Z$.
Then the rational map $\gamma: Z^\prime\dasharrow Z^{\prime\prime}$ is
an isomorphism in codimension one and induces an isomorphism
$\gamma^*:H^2(Z^{\prime\prime},\IR)\xra{\sim}H^2(Z^{\prime},\IR)$.
Assume that $\lambda\in\Amp(\beta^{\prime})\cap\gamma^*(\Amp(\beta^{\prime\prime}))$
is a class corresponding to line bundles
$L^\prime\in\Pic(Z^\prime/Z)$ and
$L^{\prime\prime}\in\Pic(Z^{\prime\prime}/Z)$. Replacing $L^\prime$ and $L^{\prime\prime}$ by
a sufficiently high power, we assume that  $L^\prime$ and $L^{\prime\prime}$
are relatively very ample. For codimension reasons,
$H^0(Z^\prime,L^\prime)\isom H^0(Z^{\prime\prime},L^{\prime\prime})$.
It particular $Z^\prime$ and $Z^{\prime\prime}$ are embedded with the same
image into some projective space $\IP^n_Z$. Phrased differently,
if $\gamma: Z^\prime\dasharrow Z^{\prime\prime}$ is not a isomorphism
then the ample cones must be disjoint.

Applying this argument to the variety $\SS$, we infer from
\eqref{eq:NefCones} that any resolution of $\SS$ must be of the form
$\varphi_B:\SS_B\to\SS$.
\end{proof}

\section{Universal Poisson deformations}\label{section:UniversalPoissonDeformations}

We now go back to the situation of the preceding sections and consider for a given Borel
subgroup $B\supset T$ the simultaneous resolution of \eqref{eq:SlodowySliceCD}
\begin{equation}
\begin{CD}
\SS_B @>{\pi_B}>> \SS \\
@V\varphi_{_{B\SS}} VV @VV{\varphi_\SS}V \\
\h @>>> \h/\Wgp
\end{CD}
\end{equation}
We have seen in section \ref{section:SlodowySlices} that $S$ has a natural Poisson structure
relative to $\h/W$ as well as $\SS_B$ relative to $\h$. 

We start proving Theorem \ref{th:mainTheoremSimplyLaced}. The strategy is to show that the
commutative diagram of Slodowy slices coincides with the abstract one construted in section
\ref{section:ReviewPoissonDeformations}.

Let us consider $\varphi_{_{B\SS}} : \SS_B \to \h$ and denote by $\SS_{B,0}$ the central
fibre of $\varphi_{_{B\SS}}$. Equivalence classes of infinitesimal Poisson
deformations of $\SS_{B,0}$ are
classified by the 2-nd hypercohomology of the Lichnerowicz-Poisson complex
\cite[Prop.\ 8]{Namikawa-Flops}
$$ \wedge^1 \Theta_{\SS_{B,0}} \stackrel{\delta}\to \wedge^2
\Theta_{\SS_{B,0}} \stackrel{\delta}\to \ldots, $$
where $\wedge^p \Theta_{\SS_{B,0}}$ is placed in degree $p$. Since the Poisson structure of 
$\SS_{B,0}$ comes from a symplectic structure, the Lichnerowicz-Poisson complex can be 
identified with the truncated De Rham complex
$$ \Omega^1_{\SS_{B,0}} \stackrel{d}\to \Omega^2_{\SS_{B,0}}\stackrel{d}\to\ldots$$
Using the fact that 
$$H^1(\SS_{B,0}, \mathcal{O}_{\SS_{B,0}}) = H^2(\SS_{B,0}, \mathcal{O}_{\SS_{B,0}}) = 0,$$
we see \cite[Cor.\ 10]{Namikawa-Flops} that
$$\mathbb{H}^2(\SS_{B,0}, \Omega^{\geq 1}_{\SS_{B,0}}) = H^2(\SS_{B,0},\IC).$$
In the language of the Poisson deformation functor (cf.\ section \ref{section:ReviewPoissonDeformations})
we have
$$T_{\mathrm{PD}_{\SS_{B,0}}} = H^2(\SS_{B,0},\IC).$$
The relative symplectic 2-form of $\SS_B /\h$ defines a class in
$H^2(\SS_{B,t}, \IC)$ for each fibre $\SS_{B,t}$ of $\SS_B \to \h$. Identifying
$H^2(\SS_{B,t}, \IC)$ with $H^2(\SS_{B,0}, \IC)$ we get a period map
$$ \kappa:\h \to H^2(\SS_{B,0}, \IC),$$
which can be regarded as the Kodaira-Spencer map for the Poisson deformation $\SS_{B} \to \h$.
%
%
%
%
\begin{proposition}\label{proposition:PoissonOverh} --- 
Suppose condition $(*)$ holds. Then $\varphi_{_{B\SS}}:\SS_B\to\h$
is the formal universal Poisson deformation of $\SS_{B,0}$.
\end{proposition}

\begin{proof} Consider the Kodaira-Spencer maps $\kappa:\h\to H^2(\SS_{B,0},\IC)$ and
$\kappa^\prime:\h\to T_{\mathrm{PD}_{T^*(G/B)}}=H^2(T^*(G/B),\IC)$ associated to the Poisson
deformation $\varphi_{_{B\SS}}:\SS_B\to\h$ of $\SS_{B,0}$ and $G\times^B\b\to\h$ of $T^*(G/B)$, 
respectively. As the relative symplectic structure of $\SS_B$ is the restriction of the relative 
symplectic structure of $G\times^B\b$, the following diagram is commutative
\begin{equation}
\xymatrix@=10pt{
&\h\ar[dl]_{\kappa^\prime}\ar[dr]^{\kappa}\\
H^2(T^*(G/B),\IC)\ar[rr]&&H^2(\SS_{B,0},\IC)
}
\end{equation}
We have seen in the proof of Proposition \ref{prop:Restriction} that the horizontal restriction
map is an isomorphism. By (\cite{Namikawa-PoissonDeformations2}, Proposition 2.7), the Poisson 
deformation $G\times^B\b\to\h$ is formally universal at $0\in\h$. Hence $\kappa^\prime$ is an isomorphism,
and so is $\kappa$. This means that $\varphi_{_{B\SS}}$ is formally semi-universal. By 
\cite{Namikawa-PoissonDeformations1}, Corollary 2.5, $\varphi_{_{B\SS}}$ is formally universal.
\end{proof}

We next treat $\varphi_\SS:\SS\to\h/W$. By Proposition \ref{proposition:PoissonOverh} and Theorem \ref{th:NamikawaDeformation},
the base space of the universal Poisson deformation of $\So$ can be written as $\h/H$ for some finite 
subgroup $H\subset\LieGl(\h)$. The following proposition asserts that $H$ is actually the 
Weyl group $W$ and finishs the proof of  Theorem \ref{th:mainTheoremSimplyLaced}. 

\begin{proposition}--- Suppose condition $(*)$ holds. Then $\varphi_\SS:\SS\to\h/W$
is the formally universal Poisson deformation of $\SS_0$.
\end{proposition}

\begin{proof} Let $\mathfrak{X}\to\h/H$ be the universal $\IC^*$-equivariant Poisson deformation
of $\So$. As $\underline{\Spec}(\pi_{B*}\kO_{S_B})\isom \h\times_{\h/W}S$ the morphism $h\to\h/H$ factors
through the quotient map $\h\to\h/\Wgp$ and the classifying morphism $\h/\Wgp\to\h/H$ for the Poisson deformation
$S\to\h/\Wgp$ of $\So$. (Strictly speaking the morphisms $\h\to\h/\Wgp\to\h/H$ are first
defined on the respective completions of the origin. As $\IC^*$ acts with strictly positive weights on all three
spaces and all maps are equivariant, the morphisms are defined as stated above.)

Consider the commutative diagram
\begin{equation}
\begin{CD}
\SS_B @>{\pi_B}>> \SS @>>> \mathfrak{X}\\
@VVV @VVV @VVV\\
\h @>>> \h/\Wgp   @>>> \h/H.
\end{CD}
\end{equation}
It follows from the factorisation  $\h\to\h/\Wgp\to\h/H$ that $\Wgp\subset H$. We need to prove the converse inclusion.

Let $h\in H$. Define $S^\prime$ to be the fibre product
\begin{equation}
\begin{CD}
S^\prime @>h^\prime>\sim>\SS_B \\
@VVV @VVV \\
\h\times_{\h/H}\mathfrak{X}@>{h\times\id}>\sim>\h\times_{\h/H}\mathfrak{X}
\end{CD}
\end{equation}
Note that $\h\times_{\h/H}\mathfrak{X} = \h\times_{\h/W}S$. Since $\varphi^\prime$ is a crepant resolution, by
Theorem \ref{th:CrepantResolutions} there exists an element $w\in\Wgp$ and a commutative diagram
\begin{equation}
\xymatrix@=10pt{
\SS_{w(B)}\ar[rr]^{q}_\sim\ar[ddr]&&S^\prime\ar[ddl]\\
\\
&\h\times_{\h/H}\mathfrak{X}
}
\end{equation}
On the other hand, $w$ defines a map
\begin{equation}
w^\prime:G\times^B\b\to G\times^{w(B)}\ad(w)(\b),\quad [g,\xi]\mapsto[gw^{-1},\ad(w)(\xi)].
\end{equation}
and induces a cartesian diagram
\begin{equation}
\begin{CD}
\SS_B @>w^\prime>\sim>\SS_{w(B)} \\
@VVV @VVV \\
\h\times_{\h/H}\mathfrak{X}@>{w\times\id}>\sim>\h\times_{\h/H}\mathfrak{X}
\end{CD}
\end{equation}
By construction, $h^\prime q w^\prime:\SS_B\to\SS_B$ induces the identity on the central fibre
$\phi_B^{-1}(0)$, hence the composite diagram
\begin{equation}
\begin{CD}
\SS_B @>h^\prime q w^\prime>\sim>\SS_{B} \\
@V\varphi_{B} VV @VV\varphi_{B} V \\
\h@>{hw}>\sim>\h
\end{CD}
\end{equation}
is an automorphism of the universal Poisson deformation of $\SS_{B,0}$. It follows that $hw=\id$, hence $h\in\Wgp$.
\end{proof}

When $(*)$ is not satisfied, $\dim \h < b_2(\SS_{B,0},\IC)$. Since the base spaces of universal Poisson deformations of $\SS_{B,0}$ 
and $S_0$ should have dimension $b_2(\SS_{B,0},\IC)$, none of $\varphi_{_{B\SS}}:\SS_B\to\h$ or $\varphi_\SS:\SS\to\h/W$ is 
universal when $(*)$ does not hold. 
 
We will investigate when $(*)$ is satisfied in the following sections.

\section{The simply laced case}

It remains to prove the following proposition.

\begin{proposition}\label{prop:ConditionStarForSimplyLacedg} --- 
Condition $(*)$ holds for all non-regular nilpotent element $x$ in
a simple Lie algebra of type $A,D,E$.
\end{proposition}

\begin{proof} Note that in diagram \eqref{eq:Dreiecksdiagram} the map $a_p$ is an 
isomorphism by (\cite{Slodowy-Utrecht}, 4.5), hence $a_x$ is an isomorphism as well.
\end{proof}

\section{Lie algebras of type $C$}

Let $\g$ be a classical Lie algebra of type $B_n$ or $C_n$. The nilpotent orbits may by labelled by
partitions $d=[d_1\geq d_2\geq \dots \geq d_\ell]$ of $2n+1$ resp.\ $2n$ corresponding to the Jordan type
of the elements in the orbit. For Lie algebras of type $B$, \ie $\g=\Lieso_{2n+1}$, the map
$d\mapsto\kO_d$ defines a  bijection between the set of those partitions where each even part
appears with even multiplicity and the set of nilpotent orbits (\cite{CollingwoodMcGovern},
Theorem 5.1.2). Similarly, for Lie algebras of type $C$, \ie $\g=\Liesp_{2n}$, the map $d\mapsto\kO_d$
defines a  bijection between the set of those partitions where each odd part appears with even multiplicity
and the set of nilpotent orbits (\cite{CollingwoodMcGovern}, Theorem 5.1.3). The set of partitions is
partially ordered by the rule
\begin{equation}
d\leq d^\prime\quad:\Leftrightarrow\quad\sum_{i=1}^k d_i\leq
\sum_{i=1}^k d^\prime_i\text{ for all }k.
\end{equation}
By a theorem of Gerstenhaber and Hesselink (\cite{CollingwoodMcGovern}, Theorem 6.2.5)
$\kO_d\subset\overline{\kO_{d^\prime}}$ if and only if $d\leq d^\prime$.

In this section, we study Slodowy slices for the Lie algebra $\g=\Liesp(2n)$ for $n\geq 2$.

\begin{proposition}\label{prop:ConditionStarForC} ---
Let $d$ be the Jordan type of a non-regular nilpotent orbit in $\Liesp_{2n}$, and let
$F_d$ denote the isomorphism type of the fibre $\pi^{-1}(x)$ over
an element $x\in\kO_d$ under the Springer resolution $\pi:G\times^B\b\to\g$.
\begin{enumerate}
\item $\dim H^2(F_d,\IQ)=n+1$ for $d=[n,n]$ and $d=[2n-2i,2i]$,  where $i=1,\ldots,\lfloor \tfrac n 2 \rfloor$.
\item For all other $d$, $\dim H^2(F_d,\IQ)=n$.
\end{enumerate}
\end{proposition}

As a consequence, condition $(*)$ holds for all non-regular nilpotent orbits except
$d=[n,n]$ or $d=[2n-2i,2i]$ for some $i=1,\ldots,\lfloor \tfrac n 2 \rfloor$. This proves part
$(C_n)$ of Theorem \ref{th:mainTheoremNonSimplyLaced}.
In the proof, we will use the following interpretation of the fibres of the Springer resolution:
The fibre $F_d=\pi^{-1}(x)$ can be identified with the set of Borel subalgebras in $\g$
containing $x$. Borel algebras $\b$ in the symplectic Lie algebra $\Liesp_{2n}$ correspond in turn to
flags $[W]\;:\;W_1\subset W_2\subset\ldots W_n\subset \IC^{2n}$ of isotropic subspaces, and the condition
$x\in\b$ translates into $x(W_i)\subset W_{i-1}$ and $x(W_1)=0$. Sending a flag $[W]$ to the line $W_1$
defines a morphism $\psi:F_{[n,n]}\to \IP(\Ker(x))$. For a given line $[W_1]\in\IP(\Ker(x))$, we may
consider the symplectic vector space $\overline V=W_1^\perp/W_1\isom\IC^{2n-2}$ with the induced
nilpotent endomorphism $\bar x$. If $\overline d$ denotes the Jordan type of $\overline x$, then the
fibre $\psi^{-1}([W_1])$ is isomorphic to $F_{\overline d}$. This observation will allow us to argue
by induction on the dimension of the symplectic space.

\begin{proof} For any partition $d$ of length $\geq 3$, we have $d\leq [2n-2,1,1]$, whereas
for any partition $d$ of length $2$, we have $[n,n]\leq d$.

1. We first consider the case of a partition $d=[2n-k,k]$ with two parts. As the multiplicity
of each odd part needs to be even, this leaves the cases $[2n-2i,2i]$ and, if $n$ is odd, in addition
the partition $[n,n]$.

The subregular nilpotent orbit has Jordan type $[2n-2,2]$.   
It is well-known (cf. \cite{Slodowy-LectureNotes}) that the fibre $F_{[2n-2,2]}$
is the union of $n+1$ projective lines
intersecting in a $D_{n+1}$-configuration, so that $\dim H^2(F_{[2n-2,2]},\IQ)=n+1$.
Moreover, the ordering of the partitions of length 2 takes the form
$[2n-2,2]>[2n-4,4]>\ldots>[n,n]$. By Lemma \ref{lemma:RestrictionIsInjection} it suffices to
prove assertion 1 of the proposition for the partition $[n,n]$. The case $n=2$ is already covered,
and we proceed by induction. We work with the following model for a $2n$-dimensional symplectic
vector space with a nilpotent
endomorphism of type $[n,n]$: Let the action of $x$ on $V=\langle v_1,\ldots,v_n,w_1,\ldots,w_n\rangle$
be given by $xv_i=v_{i+1}$, $xw_i=w_{i+1}$ and $xv_n=xw_n=0$, and let the symplectic form be
fixed by $\omega(v_i,w_j)=(-1)^{i-1}\delta_{i+j,n+1}$ and $v_i\perp v_j$, $w_i\perp w_j$ for all
$i$, $j$. Consider now the map $\psi:F_{[n,n]}\to \IP(\Ker(x))=\IP^1$ sending a flag $[W]$ of
isotropic subspaces $W_1\subset\ldots\subset W_n$ to the line $W_1\subset \Ker(x)$. In order
to identify the fibre $\psi^{-1}(W_1)$ we need to determine the Jordan type of the induced
endomorphism $\overline x$ on $\overline V=W_1^\perp/W_1$.

Assume first that $n$ is odd. The kernel of $x$ is spanned by $v_n$ and $w_n$, and if $W_1=\langle
\alpha v_n+\beta w_n\rangle$, then $W_1^\perp$ is spanned by the elements
$\alpha v_i+\beta w_i$, $i=1,\ldots,n$ and $\gamma v_i+\delta w_i$, $i=2,\ldots,n$, where $\gamma,\delta$
are constants with $\alpha\delta-\beta\gamma\neq 0$. Then $\overline V$ is spanned by the classes
of $\alpha v_i+\beta w_i$, $i=1,\ldots,n-1$ and $\gamma v_i+\delta w_i$, $i=2,\ldots,n$. As the action
of $\overline x$ can be easily read off from this bases we see that $\overline x$ has Jordan type
$\overline d=[n-1,n-1]$. More precisely, this description shows that $\psi:F_{[n,n]}\to\IP^1$ is
a fibre bundle with fibres $F_{[n-1,n-1]}$. In particular, $R^2\psi_*\IQ$
is a local system on $\IP^{1,an}$, and by induction its rank is $n$. Since $\IP^1$ is simply connected,
this local system is trivial. Finally, all fibres have vanishing first cohomology groups and the Leray
spectral sequence $H^p(\IP^1, R^q\psi_*\IQ)\implies H^{p+q}(F_{[n,n]},\IQ)$ shows that
$H^2(F_{[n,n]},\IQ)\isom \IQ^{n+1}$.

Assume now that $n$ is even. We keep the notation of the previous paragraph. However, in the even case
the structure of the fibre $\psi^{-1}(\langle \alpha v_n+\beta w_n\rangle)$ does depend on the choice
of $W_1=\langle\alpha v_n+\beta w_n\rangle$. If $\alpha\neq 0\neq \beta$, then $W_1^\perp/W_1$ is generated
by $\alpha v_i+\beta w_i$, $i=2,\ldots,n-1$, and $\alpha v_i-\beta w_i$, $i=1,\ldots,n$. It follows that
$\psi^{-1}(U)\to U:=\IP^1\setminus\{[1:0],[0:1]\}$ is a fibre bundle (in the Zariski topology) with
fibres $F_{[n,n-2]}$. On the other hand, if $W_1=\langle v_n\rangle$, then $W_1^\perp/W_1$ is generated
by $v_i$, $i=1,\ldots,n-1$ and $w_i$, $i=2,\ldots,n$, so that $\psi^{-1}([1:0])\isom F_{[n-1,n-1]}$.
Similarly, one shows that $\psi^{-1}([0:1])\isom F_{[n-1,n-1]}$. By induction, we know that the second
Betti number both of $F_{[n-1,n-1]}$ and $F_{[n,n-2]}$ equals $n$. We claim that $R^2\psi_*\IQ$ is a
local system on $\IP^{1,an}$ of rank $n$; we can then finish the argument as in the odd case.

In order to prove the claim, let $\kw_1\subset \ko_{\IP^1}\tensor \Ker(x)\subset \ko_{\IP^1}\tensor \IC^{2n}$
denote the tautological
subbundle and let $\overline\kv=\kw^\perp_1/\kw_1$ denote the corresponding quotient bundle, endowed with
a nilpotent endomorphism $\overline x$. For any point $p\in\IP^1$ there are sections
$s_1,\ldots,s_{2n-2}$ in $\overline\kv$ defined over an open neighbourhood $U$ that form a standard
basis for the symplectic structure $\overline\omega$ on $\overline\kv$. The matrix coefficients
$\overline x(s_i)=\sum_j x_{ji}s_j$ define a classifying morphism $f:U\to \Liesp_{2n-2}$, and we
obtain a diagram with cartesian squares:
$$
\begin{array}{ccccc}
F_{[n,n]}&\supset&\psi^{-1}(U)&\lra& \LieSp_{2n-2}\times^{\bar B}\bar{\b}\\
\Big\downarrow&&\Big\downarrow&&\Big\downarrow\\
\IP^1&\supset&U&\lra&\Liesp_{2n-2}
\end{array}
$$
If $n\geq 4$, as we assume, the classifying morphism $f$ takes value in the closure of the subregular
nilpotent orbit. Let $\xi_1,\ldots,\xi_n$ denote a basis of the germ $(R^2\psi_*\IQ)_p=
H^2(\psi^{-1}(p),\IQ)\isom \IQ^n$. On an appropriate open neighbourhood $U'$ of $p$, this basis
forms a set of linearly independent vectors in $(R^2\psi_*\IQ)_{q}$ for all $q\in U'$ by Lemma
\ref{lemma:RestrictionIsInjection}. As the rank of the second Betti
number is constant, they form in fact a basis. This shows that $R^2\psi_*\IQ$ is indeed a local system.

2. Consider now the opposite case of a partition $d$ with at least $3$ parts. Any such partition is
dominated by $[2n-2,1,1]$, and it suffices to show that $\dim H^2(F_{[2n-2,1,1]},\IQ)=n$.
A model for a nilpotent symplectic endomorphism of Jordan type $[2n,1,1]$ is given as follows:
Let $V=\langle v_1,\ldots,v_{2n-2},u_1,u_2\rangle$ with $xv_i=v_{i+1}$ and $x v_{2n-2}=xu_1=xu_2=0$.
The symplectic structure is defined by $\omega(v_i,v_j)=(-1)^i \delta_{i+j,2n-1}$ and $\omega(u_1,u_2)=1$,
the other matrix entries of $\omega$ being $0$. In this case, $\psi:F_{[2n-2,1,1]}\to\IP(\Ker\,x)=\IP^2$
is an isomorphism over $\IP^2\setminus \{[v_{2n-2}]\}$ with $\psi^{-1}([v_{2n-2}])\isom F_{[2n-4,1,1]}$.
More precisely, the rational map $\psi^{-1}$ extends to an embedding of the blow-up $\Blowup_{[v_{2n-2}]}(\IP^2)$
into $F_{[2n-2,1,1]}$. By induction one obtains the following description: $F_{[2n-2,1,1]}$ is
the union $S_1\cup S_2\cup\ldots\cup S_{n-1}$ of $(n-1)$ Hirzebruch surfaces $S_i\isom \IF_1$, $i=1\ldots,n-1$,
where $S_k$ and $S_{k+1}$ are glued along a line which has self-intersection $+1$
in $S_k$ and self-intersection $-1$ in $S_{k+1}$. It is clear from this description that
$H^2(F_{[2n-2,1,1]},\IQ)=n$.
\end{proof}

It remains to exhibit the hypersurface in the Slodowy slice to a nil\-po\-tent orbit
in $\Liesp_{2n}$ of type $[2n-2,1,1]$. We will give matrices and formulae for $n=4$, the
general pattern can easily be derived from these data: The following matrices are nilpotent

\begin{scriptsize}
$$
x=\left(\begin{matrix}
0 & 1 & 0 & 0 & 0 & 0 & 0 & 0\\
0 & 0 & 2 & 0 & 0 & 0 & 0 & 0\\
0 & 0 & 0 & 3 & 0 & 0 & 0 & 0\\
0 & 0 & 0 & 0 & 4 & 0 & 0 & 0\\
0 & 0 & 0 & 0 & 0 & 5 & 0 & 0\\
0 & 0 & 0 & 0 & 0 & 0 & 0 & 0\\
0 & 0 & 0 & 0 & 0 & 0 & 0 & 0\\
0 & 0 & 0 & 0 & 0 & 0 & 0 & 0
\end{matrix}\right),\quad
y=\left(\begin{matrix}
0 & 0 & 0 & 0 & 0 & 0 & 0 & 0\\
5 & 0 & 0 & 0 & 0 & 0 & 0 & 0\\
0 & 4 & 0 & 0 & 0 & 0 & 0 & 0\\
0 & 0 & 3 & 0 & 0 & 0 & 0 & 0\\
0 & 0 & 0 & 2 & 0 & 0 & 0 & 0\\
0 & 0 & 0 & 0 & 1 & 0 & 0 & 0\\
0 & 0 & 0 & 0 & 0 & 0 & 0 & 0\\
0 & 0 & 0 & 0 & 0 & 0 & 0 & 0
\end{matrix}\right)
$$
\end{scriptsize}
and belong to the Lie algebra $\Liesp_{2n}$ with respect to the skew-symmetric form

\begin{scriptsize}
$$J=
\left(\begin{matrix}
0 & 0 & 0 & 0 & 0 & -1 & 0 & 0\\
0 & 0 & 0 & 0 & \frac{1}{5} & 0 & 0 & 0\\
0 & 0 & 0 & - \frac{1}{10} & 0 & 0 & 0 & 0\\
0 & 0 & \frac{1}{10} & 0 & 0 & 0 & 0 & 0\\
0 & - \frac{1}{5} & 0 & 0 & 0 & 0 & 0 & 0\\
1 & 0 & 0 & 0 & 0 & 0 & 0 & 0\\
0 & 0 & 0 & 0 & 0 & 0 & 0 & -1\\
0 & 0 & 0 & 0 & 0 & 0 & 1 & 0
\end{matrix}\right)$$\end{scriptsize}
The entries are inverses of binomial coefficients. $x$ has Jordan type $[2n-2,1,1]$,
and the Slodowy slice to the orbit of $x$ is parameterised by the matrix

\begin{scriptsize}
$$s:=\left(\begin{array}{cccccc|cc}
0 & 1 & 0 & 0 & 0 & 0 & 0 & 0\\
5\, t_1 & 0 & 2 & 0 & 0 & 0 & 0 & 0\\
0 & 4\, t_1 & 0 & 3 & 0 & 0 & 0 & 0\\
10\, t_2 & 0 & 3\, t_1 & 0 & 4 & 0 & 0 & 0\\
0 & 4\, t_2 & 0 & 2\, t_1 & 0 & 5 & 0 & 0\\
t_3 & 0 & t_2 & 0 & t_1 & 0 & b & a\\\hline
- a & 0 & 0 & 0 & 0 & 0 & y & - z\\
b & 0 & 0 & 0 & 0 & 0 & x & - y
\end{array}\right),$$
\end{scriptsize}
depending on parameters $a,b,x,y,z,t_1,\ldots,t_{n-1}$. The characteristic polynomial of $s$
has the form
$$\chi_s(\lambda)=(2n-3)! (a^2x+2aby+b^2z)+(\lambda^2+xz-y^2)\chi_{s'}(\lambda),$$
where $s'$ is upper left block of the matrix $s$. Solving the coefficients of the terms
$\lambda^{2n-2i}$, $i=1,\ldots,n$, recursively for the variables $t_1,\ldots,t_{n-1}$ finally
leaves the equation
$$f=(2n-3)! (a^2x+2aby+b^2z)-(xz-y^2)^{n}$$
Up to a rescaling of the coordinates, this is the formula given in Example \ref{ex:hypersurfaces1}.

\section{Lie algebras of type $B$}

In this section, we consider the Lie algebra $\g=\Lieso(2n+1)$ for $n\geq 2$.

\begin{proposition}\label{prop:ConditionStarForB}--- Let $d$ be the Jordan type of a non-regular
nilpotent orbit $\kO_d$ in $\Lieso_{2n+1}$, and let $F_d$ denote the isomorphism type of the fibre
$\pi^{-1}(x)$ over an element $x\in\kO_d$ under the Springer resolution $\pi:G\times^B\b\to\g$.
\begin{enumerate}
\item $\dim H^2(F_d,\IQ)=2n-1$ for $d=[2n-1,1,1]$, \ie the subregular orbit.
\item For all other $d$, $\dim H^2(F_d,\IQ)=n$.
\end{enumerate}
\end{proposition}

The proposition implies part $(B_n)$ of Theorem \ref{th:mainTheoremNonSimplyLaced}. As in the
case of Lie algebras of type $C$, we may identify points in $F_d$ with flags $W_1\subset
\ldots\subset W_n\subset\IC^{2n+1}$ of isotropic subspaces. We use the same techniques and
notations as in the previous section. As the proof proceeds along similar lines we only indicate
the differences.

\begin{proof} As $\Lieso_5\isom\Liesp_{4}$, the case $n=2$ is already covered
by Proposition \ref{prop:ConditionStarForC}. We may therefore assume that $n\geq 3$ and proceed
by induction. 


For the subregular orbit the transversal slice $S_0=S\cap N$ is a surface singularity of type
$A_{2n-1}$ (cf. \cite{Slodowy-Utrecht}, Section 1.8) so that the fibre $F_{[2n-1,1,1]}$ is the 
union of $2n-1$ copies of $\IP^1$ and its second Betti number is $2n-1$.
By Lemma \ref{lemma:RestrictionIsInjection} it suffices to prove the assertion for the next 
smaller orbit of type $[2n-3,3,1]$.

Let $v_1,\ldots,v_{2n-3}$, $v_1',v_2',v_3'$, $v''$ be a basis of $V=\IC^{2n+1}$, let the action
of $x$ be given by $xv_i=v_{i+1}$, $xv_i'=v'_{i+1}$ and $xv_{2n-3}=xv_3'=xv''=0$, and let
the quadratic form on $V$ be given by $b(v_i,v_j)=(-1)^{i-1} \delta_{i+j,2n-2}$, $b(v_i',v_j')=
(-1)^{i-1}\delta_{i+j,4}$ and $b(v'',v'')=1$, all other matrix coefficients of $b$ being $0$.
Then $\Ker(x)=\langle v_{2n-3}, v'_3,v''\rangle$. A vector $\alpha v_{2n-3}+\beta v_3'+\gamma v''$
is isotropic if and only if $\gamma=0$. Thus sending a flag to its one-dimensional part defines
a morphism $F_{[2n-3,3,1]}\to \IP\langle v_{2n-3},v'_3\rangle=\IP^1$. Let $W_1=\langle \alpha
v_{2n-3}+\beta v_3'\rangle$. If $[\alpha:\beta]=[1:0]$, then the induced endomorphism $\overline x$
on $W_1^\perp/W_1$ is of type $[2n-5,3,1]$. If on the other hand $\beta\neq 0$, then the
induced endomorphism is of type $[2n-3,1,1]$. It follows by induction that
$R^2\psi_*(\IQ)_{[1:0]}\isom \IQ^{n-1}$ and that the restriction of $R^2\psi_*(\IQ)$ to
$\IP^1\setminus\{[1:0]\}$ is a trivial local system of rank $2n-3$. Taking global sections, we
find that $\dim H^0(R^2\psi_*\IQ)=n-1$, and the Leray spectral sequence gives the desired
value $b_2(F_{[2n-3,3,1]})=n$.
\end{proof}

\section{The Lie algebra of type $G_2$}

The Lie algebra $\g_2$ has five nilpotent orbits: the regular, the subregular, the `subsubregular',
the minimal and the trivial orbit of dimensions $12$, $10$, $8$, $6$ and $0$, respectively. The
orbit closures are linearly ordered by inclusion.

\begin{proposition}\label{prop:ConditionStarForG2}--- Let $F_x$ denote the fibre of the
Springer resolution for a nilpotent element $x$ in $\g_2$. Then $\dim H^2(F_x,\IQ)=4,3$ and $2$,
if $x$ belongs to the subregular, the subsubregular and the minimal orbit, respectively.
\end{proposition}

In the course of the proof we will also establish the formula for the polynomial $f$ in 
Example \ref{ex:hypersurfaces2}.

\begin{proof} Similarly to the previous section, the transversal slice to the subregular orbit
is a surface singularity, this time of type $D_4$ (cf. \cite{Slodowy-Utrecht}, Section 1.8), so 
that the Springer fibre is the union of $4$ copies of $\IP^1$ and its second Betti number is $4$. 

The slice to the minimal orbit
in the nilpotent cone will be shown to be a $6$-dimensional symplectic hypersurface. We will
determine the second Betti number of the fibre via an explicit resolution of the singularities.

In order to facilitate calculation in the Lie algebra $\g_2$ we will use the $\IZ/3$-grading of 
$\g_2$ with graded pieces $\g_2^{(0)}=\Liesl_3$, $\g_2^{(1)}=\IC^3$, and $\g_2^{(2)}=\IC^{3*}$, 
where $\IC^3$ and $\IC^{3*}$ denote the standard representation of $\Liesl_3$ and its dual (see 
for example \cite{FultonHarris}, 22.2). Representing the latter as column and row vectors, 
respectively, we may write a general element of $\g_2$ in block form as
$$
\left(\begin{array}{c|c}
A&v\\\hline w&0
      \end{array}\right)
=
\left(\begin{array}{ccc|c}
h_1&a_{12}&a_{13}&v_1\\
a_{21}&h_2-h_1&a_{23}&v_2\\
a_{31}&a_{32}&-h_2&v_3\\
\hline
w_1&w_2&w_3&0
\end{array}\right)
$$
The Lie bracket is given by the canonical maps $\Liesl_3\times \Liesl_3\to \Liesl_3$,
$\Liesl_3\times\IC^3\to\IC^3$, $\Liesl_3\times \IC^{3*}\to \IC^{3*}$, 
$\IC^3\times \IC^3\xra{\wedge} \IC^{3*}$, $\IC^{3*}\times\IC^{3*}\xra{\wedge}\IC^3$ and
$$\IC^3\times \IC^{3*}\to \Liesl_3, (v,w)\mapsto \frac34 \left(vw- \frac13 (wv) I\right),$$
where the factor $\frac34$ is thrown so as to make the Jacobi identity hold. In this notation,
the representation $\rho:\g_2\to \Lieso_7$ can be written as
$$\rho:\left(\begin{array}{c|c}
A&v\\\hline w&0
      \end{array}\right) \mapsto \left(
\begin{array}{ccc}
A&\frac1{\sqrt 2}v& M(w^t)\\

-\frac1{\sqrt 2}w&0&-\frac1{\sqrt{2}}v^t\\

M(v)&\frac1{\sqrt 2}w^t&-A^t
\end{array}\right)$$
where $M(v)$ is the $3\times 3$-matrix of the linear map $u\mapsto v\times u$, the vector cross product.
The two components of the coadjoint quotient map $\chi=(\chi_2,\chi_6):\g_2\to \IC^2$ are the
coefficients of $t^5$ and $t$, respectively, in the characteristic polynomial $u\mapsto \det(tI-\rho(u))$,
and are of degree $2$ and $6$, respectively. They can be expressed in terms of $\Liesl_3$-invariants as follows:
\begin{eqnarray*}
 \chi_2&=&\frac32 wv-\tr(A^2),\\
\chi_6&=&- \det(A)^2 + \frac32\det(A)(wAv) + \frac3{16}(wAv)^2+\frac14\tr(A^2)^2 (wv)\\
&&+ \frac14\tr(A^2)(wv)^2 - \frac12(wA^2v)\tr(A^2) + \frac1{16}(wv)^3 - \frac34(wA^2v)(wv)\\
&& + \frac12\det(v|Av|A^2v) - \frac12\det(w^t|(wA)^t|(wA^2)^t)
\end{eqnarray*}
Consider now the $\Liesl_2$-triplet
$$x=\left(\begin{array}{cccc}0&1&0&0\\0&0&0&0\\0&0&0&0\\0&0&0&0\end{array}\right),\quad
h=\left(\begin{array}{cccc}1&0&0&0\\0&-1&0&0\\0&0&0&0\\0&0&0&0\end{array}\right),\quad
y=\left(\begin{array}{cccc}0&0&0&0\\1&0&0&0\\0&0&0&0\\0&0&0&0\end{array}\right)
$$
The Slodowy slice to $x$ consists of all elements of the form
$$\xi:=\left(\begin{array}{ccc|c}
\tfrac12 b& 1& 0& 0\\ u&\tfrac12 b& p& 2q\\s& 0& -b& a\\
\hline 2r& 0& c& 0
\end{array}\right)
$$
The coordinates $a,b$ and $c$ have degree $2$, whereas $p,q,r,s$ have degree $3$ and $u$ has
degree $4$. Calculation shows
$$\chi_2(\xi)=-2(u-\tfrac34(ac-b^2)).$$
In particular, the derivative of $\chi|_S$ has rank $1$. Modulo $\chi_2$, we can express $u$ in terms
of the coordinates $a,b,c$. This simplifies considerably the expression for $\chi_6$. In fact,
\begin{eqnarray*}\chi_6(\xi)&=&t_1t_3-t_2^2-\tfrac12 (t_4^2a+2t_4t_5b+t_5^2c)\\
&&-\tfrac12(at_3-2bt_2+ct_1)\chi_2(\xi)+\tfrac14(ac-b^2) \chi_2(\xi)^2,
\end{eqnarray*}
where
$$\begin{array}{rclcrcl}
t_1&=&a(ac-b^2)+2(q^2-rp),&\quad&z_1&=&as-2br+cq,\\[1ex]
t_2&=&b(ac-b^2)+(rq-ps),&&z_2&=&ar-2bq+cp,\\[1ex]
t_3&=&c(ac-b^2)+2(r^2-qs).\\
\end{array}
$$
Note that $\chi^{-1}(0)\subset\IC^8$ is isomorphic to the hypersurface $X=\{f=0\}\subset\IC^7$
with $f=z_1^2a-2z_1z_2b+z_2^2c+2(t_2^2-t_1t_3)$. The polynomials $t_1$, $t_2$, $t_3$, $z_1$, $z_2$
are a minimal set of equations for the reduced singular locus $\Sigma$ of $X$. 

Let $\pi':X'\to X$ denote the blow-up of $X$ along $\Sigma$. As $X$ and $\Sigma$ are defined
by explicitly given polynomials, the calculation can be done with a computer algebra
system. We used the program SINGULAR \cite{Singular}. Inspection of the five affine coordinate charts
for $X'$ shows that $\pi$ is semismall, that the singular locus $\Sigma'$ of $X'$ is irreducible and 
smooth, and that $X'$ has transversal $A_1$-singularities along $\Sigma'$. A second blow-up 
$\pi'':X''\to X'$ along $\Sigma'$ finally yields a symplectic resolution of $X$. Let 
$F':=(\pi')^{-1}(0)$. Then $F'$ is the hypersurface in $\IP^4=\IP(x_0,x_1,x_2,y_0,y_1)$ given by 
the equation $x_0x_2-x_1^2$. The intersection of $\Sigma'$ with $F'$ equals set-theoretically the 
singular locus of $F'$. It is however defined by the equations $x_0+2x_1y_1+x_2y_1^2$, $x_2(x_1+2y_1)$, 
$x_2^3$ on the affine chart $U_0=\{y_0=1\}$, and by analogous equations on the affine chart 
$U_1=\{y_1=1\}$. Note that the singular locus of $F'$ is covered by $U_1\cup U_2$. For an appropriate 
coordinate change, we
get $F'\cap U_1=\IA^1\times Z$ with $Z=\{u_1u_3-u_2^2\}\subset\IA^3$, and where the center of
the blow-up is defined by $u_1,u_2u_3, u_3^3$. The blow-up of $Z$ along this ideal has a cell
decomposition $\IA^2\sqcup \IA^1$. As the part not covered by the two charts is isomorphic to a
smooth quadric in $\IP^3$, we obtain in fact a cell decomposition
of $F'':=(\pi''\circ\pi')^{-1}(0)=(\pi'')^{-1}(F')$:
$$F''=(\IA^0\sqcup \IA^1)\sqcup (\IA^1\sqcup \IA^0) \times (\IA^2\sqcup \IA^1)= \IA^0\sqcup 2 \IA^1
\sqcup 2\IA^2 \sqcup \IA^3.$$
In particular, $H^2(F'',\IZ)=\IZ^2$. As $F''$ and the fibre $F_x$ have the same Betti numbers, this proves
the proposition for the minimal orbit.

Similar techniques can be used to treat the subsubregular orbit. It turns out that the slice is
again a hypersurface, which is in fact isomorphic to the hypersurface slice to the orbit of
type $[4,1,1]$ for $\Liesp_6$. In particular, the Betti number of the fibre is 3 in this case.
\end{proof}

\section{The Lie algebra of type $F_4$}

\begin{proposition}\label{prop:F4Case}--- Let $F_x$ denote the Springer fibre of a nilpotent
element $x$ in $\gothf_4$. Then $\dim H^2(F_x,\IQ)=0,6$ or $4$ if $x$ belongs to the
regular, the subregular or any other orbit, respectively.
\end{proposition}

For the discussion of the Lie algebra $\gothf_4$ we follow a path different from the
other Lie algebras. In \cite{DLP}, De Concini, Lusztig and Procesi describe a general
method of how to construct a partition of the Springer fibre of a nilpotent element into
locally closed subvarieties. In the case of classical Lie algebras, this partition
actually yields a cell decomposition. For the exceptional Lie algebras the situation
is more complicated. Fortunately, the situation simplifies when the nilpotent element
in question is {\sl distinguished} in the sense of Bala and Carter \cite{BalaCarter}.
This is the case for a nilpotent element in the subsubregular orbit (= 3-rd largest 
nilpotent orbit) of the Lie algebra $\gothf_4$.
We outline the method of \cite{DLP}:

Let $x\in\gothg$ be a nilpotent element in a simple Lie algebra. By the Jacobson-Morozov
theorem, there are elements $h,y\in\gothg$ such that $x,h,y$ form a standard
$\Liesl_2$-triplet in $\gothg$. One can choose a Cartan subalgebra $\gothh\subset\gothg$
containing $h$ and a root basis $\Delta=\{\alpha_i\}_{i=1,\ldots\ell}$ such that
$\alpha_i(h)\geq 0$ for all $i=1,\ldots,\ell$. In fact, by a result of Dynkin, $\alpha_i(h)\in
\{0,1,2\}$. Thus one may associate to $x$ a weighted Dynkin diagram, where the node corresponding
to the root $\alpha_i$ is labelled by $\alpha_i(h)$. Following Dynkin, associating to $x$
its weighted Dynkin diagram gives an injective map from the set of conjugacy classes of
nilpotent elements in $\gothg$ to the set of Dynkin diagrams labelled with numbers $0$, $1$, or $2$.
For instance, the unique nilpotent orbit in $\gothf_4$ of dimension 44, to which we will
refer as the subsubregular orbit in the following, belongs to the weighted Dynkin diagram
\unitlength=1mm
\begin{equation}\label{eq:F4DynkinDiagram}
\begin{picture}(40,0)
\put(10,0){
\put(2.5,1.3){\line(1,0){6.5}}
\put(2.5,0.7){\line(1,0){6.5}}
\put(4.5,0){$>$}
}
\put(10,0){$2$}
\put(20,0){$0$}
\put(30,0){$2$.}
\put(0,0){$0$}
\put(0,0){\put(2.5,1){\line(1,0){6.5}}}
\put(20,0){\put(2.5,1){\line(1,0){6.5}}}
\end{picture}
\end{equation}
Let $\gothg=\bigoplus_i\gothg(i)$ be the weight decomposition for the action of
$h$ on $\gothg$, i.~e.\ $\gothg(i)=\{v\in\gothg\;|\; [h,v]=iv\}$. Then $\gothg(0)$ is a reductive
subalgebra of $\gothg$, and for every $i$, the homogeneous component $\gothg(i)$ is a natural
representation of $\gothg(0)$. Moreover, by construction $x\in\gothg(2)$, and the map
$\ad(x):\gothg(0)\to \gothg(2)$ is surjective. The element $x$ is distinguished in the sense of
Bala and Carter if and only if this map is also bijective.

For the proof of the proposition, we need to understand the Lie algebra structure of $\gothf_4(0)$
and the structure of $\gothf_4(2)$ as an $\gothf_4(0)$-representation for the case of the
subsubregular
orbit in $\gothf_4$. After removal of all nodes with nonzero labels the Dynkin diagram
\eqref{eq:F4DynkinDiagram} decomposes into two $A_1$-diagrams. Let $\Liesl_2(\alpha_1)$ and
$\Liesl_2(\alpha_3)$ denote the corresponding Lie subalgebras. Then
$$\gothf_4(0)=(\Liesl_2(\alpha_1)\oplus\Liesl_2(\alpha_3))+\gothh.$$
Let $(abcd)$ denote the root space in $\gothf_4$ corresponding to the root
$a\alpha_1+b\alpha_2+c\alpha_3+d\alpha_4$. Then $\gothf_4(2)$ is the direct sum of all root spaces
$(abcd)$ with $0\cdot a+2\cdot b+0\cdot c+2\cdot d=2$. Using the explicit list in
\cite[planche VIII]{Bourbaki}, it is not difficult to see that $\gothf_4(2)$ is 8-dimensional and is
spanned by the following spaces:
$$
\begin{array}{ccc}
(0001)&\to&(0011)
\end{array}
\quad
\begin{array}{ccccc}
(0100)&\to&(0110)&\to&(0120)\\
\downarrow&&\downarrow&&\downarrow\\
(1100)&\to&(1110)&\to&(1120)
\end{array}
$$
Moreover, in this diagram, horizontal and vertical arrows denote the action of $(0010)\subset
\Liesl_2(\alpha_3)$ and $(1000)\subset\Liesl_2(\alpha_1)$, respectively. From this we see that
$$\gothf_4(2)=V_3\oplus \big(V_1\tensor S^2(V_3)\big),$$
where $V_1$ and $V_3$ denote the 2-dimensional irreducible representations of $\Liesl_2(\alpha_1)$
and $\Liesl_2(\alpha_3)$, respectively.

\begin{proof}[Proof of Proposition \ref{prop:F4Case}] According to results of Spaltenstein
\cite[table on page 250]{Spaltenstein}, all nilpotent orbits in $\gothf_4$ that are neither
regular or subregular are contained in the closure of the unique orbit of dimension 44.
It is well-known that the second Betti number of the Springer fibre for the subregular orbit is 6,
as the Springer fibre itself is an $E_6$-tree of projective lines. It therefore suffices to
show that the second Betti number of the Springer $F$ fibre for the subsubregular orbit is 4.
The same then holds for all smaller nilpotent orbits by Lemma \ref{lemma:RestrictionIsInjection}.

We may take a general element $x\in \gothf_4(2)=V_3\oplus \big(V_1\tensor S^2(V_3)\big)$ as a representative
of the subsubregular orbit. Since $x$ is distinguished, the algorithm of De Concini, Lusztig and
Procesi yields a decomposition $F=\bigcup_U F_U$ into locally closed subvarieties $F_U$, each of which is
a vector bundle $F_U\to X_U$ over a smooth subvariety $X_U\subset \kf$, the flag variety associated
to the reductive Lie algebra $\gothf_4(0)$. The index $U$ runs through the set of all linear subspaces
of $\gothf_4(2)$ that are invariant under a fixed chosen Borel subalgebra $\gothb_0\subset\gothf_4(0)$.
Note that $\kf\isom \IP(V_1)\times\IP(V_3)=\IP^1\times\IP^1$.

For a given $\gothb_0$-invariant subspace $U$, the manifolds $X_U$ are defined as follows: A point
in $\kf$ represented by a Borel subalgebra $\gothb\subset\gothf_4(0)$ is contained in $X_U$ if and
only if $x\in[\gothb,U]$. It is shown in \cite{DLP} that all $X_U$ are smooth projective varieties.
Moreover, $\codim_{\kf}(X_U)=\dim(\gothf_4(2)/U)$ and $\dim(F_U)=\dim(\kf)=2$ for all $U$. In particular,
$F_U$ cannot contribute to the second Betti number of $F$ unless $\codim_{\kf}(X_U)\leq 1$, or equivalently,
unless $U$ equals $\gothf_4(2)$ or a $\gothb_0$-invariant hyperplane therein. More precisely,
$$\dim H^2(F;\IQ)= \dim H_2( \IP^1\times\IP^1;\IQ)+\sum_{\codim U=1} \dim H_0(X_U;\IQ).$$
It remains to determine all $\gothb_0$-invariant hyperplanes $U$ in $\gothf_4(2)$ and for each $U$
the number of connected components of $X_U\subset\IP^1\times\IP^1$.

In fact, there are exactly two $\gothb_0$-invariant hyperplanes $U_1$ and $U_2$ given as follows.
If $W_1\subset V_1$ and $W_3\subset V_3$ denote the unique $\gothb_0$-invariant lines, then
$$U_1=\ker\Big(\gothf_4(2)\to V_3\to V_3/W_3\Big)$$
and
$$U_2=\ker\Big(\gothf_4(2)\to V_1\tensor S^2V_3\to V_1/W_1\tensor S^2(V_3/W_3\Big).$$
The corresponding manifolds $X_{U_1},X_{U_2}\subset \IP^1\times\IP^1$ are the zero-sets of
sections in the line bundles $\ko(0,1)$ and $\ko(1,2)$, respectively, and hence connected. This shows
that $\dim H^2(F;\IQ)=2+1+1=4$.
\end{proof}

\section{Dual pairs and Slodowy slices}\label{section:DualPairs} 
  
When $\g$ is of type $B_n$, $C_n$, $F_4$ or $G_2$, the Slodowy slice for the subregular orbit gives a   
Poisson deformation of a surface singularity of type $A_{2n-1}$, $D_{n+1}$, $E_6$ or $D_4$ respectively. 
They are not universal. On the other hand, we also have Poisson deformations of these surface singularities 
in the Lie algebras of type $A_{2n-1}$, $D_{n+1}$, $E_6$ and $D_4$. They turn out to be universal. 
It would be quite natural to expect similar phenomena for all slices listed in Theorem 1.2.   
In this section we shall consider the Poisson deformation of the (complex analytic) germ $(\So, x)$ 
instead of $\So$. Theorem \ref{th:mainTheoremSimplyLaced} holds true if we replace $S$ (resp. $\So$) by 
$(S, x)$ (resp. ($\So, x)$). One can prove the following.  

\begin{proposition}\label{prop:Dual pairs}--- Let  $\varphi_{S_1}: S_1\to {\IC}^n$ be the restriction 
of the adjoint quotient map to the Slodowy slice for $x_1 \in O_{[2n-i,i]} \subset \Lieso_{2n}$ with
$i$ odd or $i = n$, and let $\varphi_{S_2}: S_2 \to {\IC}^{n-1}$ be the restriction of the adjoint quotient 
map to the Slodowy slice for $x_2 \in O_{[2n-i-1,i-1]} \subset \Liesp_{2n-2}$. Then there are a 
hyperplane $L$ of ${\IC}^n$ and a commutative diagram of germs of complex-analytic spaces 
\begin{equation}\label{eq:TwoSlicesCD} 
\begin{CD} 
(\varphi_{S_1}^{-1}(L), x_1) @>>> (S_2, x_2) \\ 
@VVV @VVV \\ 
(L,0) @>>> ({\IC}^{n-1},0)     
\end{CD} 
\end{equation}
where the horizontal map on the first row is an isomorphism preserving the Poisson brackets up to 
a reversal of sign. In particular, the universal Poisson deformation of $(S_{2,0}, x_2)$ is realized as 
a Slodowy slice in $\Lieso_{2n}$ with the reversed Poisson structure. Here $S_{2,0}$ is the 
central fibre of $\varphi_{S_2}$.    
\end{proposition}        

\begin{proof} Let $V$ be a $2n$-dimensional complex vector space with a non-degenerate symmetric 
form $(\:, \:)_V$ and let $U$ be a $2n-2$-dimensional complex vector space with a non-degenerate
skew-symmetric form $(\:, \:)_U$. For an element $X\in\mathrm{Hom}(V,U)$, let $X^*\in\mathrm{Hom}(U,V)$ 
be its adjoint, characterised by $(Xv,u)_U = (v, X^*u)_V$. According to Kraft and Procesi \cite{K-P}, 
define maps 
$$\pi: \mathrm{Hom}(V,U) \to \Liesp(U)\quad\text{and}\quad\rho: \mathrm{Hom}(V,U) \to \Lieso(V)$$ 
by $\pi(X) := XX^*$ and $\rho(X) := X^*X$. Note that $(A, B) \in \LieSO(V) \times \LieSp(U)$ acts 
on $X \in \mathrm{Hom}(V,U)$ by $(A,B)X := B X A^{-1}$. On the other hand, $\Lieso(V)$ and $\Liesp(U)$ 
have adjoint actions of $\LieSO(V)$ and $\LieSp(U)$, respectively, and the maps $\pi$ and $\rho$ are 
$\LieSp(U)$-equivariant and $\LieSO(V)$-equivariant, respectively. Let
$\mathrm{Hom}'(V,U)\subset\mathrm{Hom}(V,U)$ be the open subset consisting of surjective linear maps,
and let $D \subset \Lieso(V)$ be the determinantal variety consisting of the endomorphisms with rank 
$\le 2n-2$. More precisely, $D$ is cut out by the vanishing of the pfaffian. By (\cite{K-P}, Theorem 1.2), 
we have 
$$\mathrm{Im}(\rho) = D.$$
Restricting $\pi$ and $\rho$ to $\mathrm{Hom}'(V,U)$ we get a diagram $$ D \stackrel{\rho'}
\longleftarrow \mathrm{Hom}'(V,U) \stackrel{\pi'}\longrightarrow sp(U).$$
Kraft and Procesi observed in (\cite{K-P}, 13.5) that one can find an element 
$X_0 \in \mathrm{Hom}'(V,U)$ so that $\pi'(X_0) \in O_{[2n-i-1,i-1]}$ and $\rho'(X_0) \in O_{[2n-i,i]}$, 
and so that both $\pi'$ and $\rho'$ are smooth at $X_0$.

Let us recall here the notion of a {\em dual pair} introduced by Weinstein \cite{W} for 
$C^{\infty}$-manifolds. Here we consider the analogous notion in the complex-analytic setting.  
A dual pair is a diagram 
$$ P_1 \stackrel{j_1}\longleftarrow P \stackrel{j_2}\longrightarrow P_2$$ 
with $P$ a holomorphic symplectic manifold, and $P_i$, $i = 1,2$, Poisson manifolds 
such that both $j_1$ and $j_2$ are Poisson mappings and $j_i^{-1}\mathcal{O}_{P_i}$ for $i = 1,2$ 
are mutual centralizers with respect to $\{\:, \:\}_P$. If $j_1$ and $j_2$ are both smooth morphisms, 
it is called a full dual pair. For instance, when a complex Lie group $G$ acts freely on $P$ preserving the
symplectic form and provided its moment map $\mu$ exists, the diagram 
$$ P/G \longleftarrow P \stackrel{\mu}\longrightarrow \g^* $$ is a dual pair (cf. \cite{W}, \S 8).  
Weinstein observed in (\cite{W}, Theorem 8.1) that if 
$$ P_1 \stackrel{j_1}\longleftarrow P \stackrel{j_2}\longrightarrow P_2$$ 
is a full dual pair, then, for any point $x \in P$, the transverse Poisson structures on $P_1$ and 
$P_2$ at $j_1(x)$ and $j_2(x)$ are anti-isomorphic as Poisson manifolds. In the remainder, we will
apply this result to the situation above. 

We define a symplectic 2-form $\omega$ on $\mathrm{Hom}(V,U)$ by $\omega(X,Y):=2tr(XY^*)$.
Then $\LieSp(U)$ and $\LieSO(V)$ naturally act on $\mathrm{Hom}(V,U)$ preserving $\omega$,
and $\pi$ and $\rho$ are the moment maps for these actions under the identifications
of $\Liesp(U)$ with $\Liesp(U)^*$ and of $\Lieso(V)$ with $\Lieso(V)^*$ by the trace maps. 
By  \cite[Proposition 11.1] {K-P}, $\LieSp(U)$ acts freely on $\mathrm{Hom}'(V,U)$ and 
$\rho'$ factorizes as  
$$\mathrm{Hom}'(V,U) \to \mathrm{Hom}'(V,U)/\LieSp(U) \subset D,$$
where $\mathrm{Hom}'(V,U)/\LieSp(U)$ is an open subset of $D$. Now let us consider the adjoint 
quotient map $\varphi: \Lieso(V) \to {\IC}^n$. 
One of the components of $\varphi$ is the pfaffian $\pf$, a square root of the determinant and hence 
an invariant polynomial of weight $n$. Define a hyperplane $L\subset {\IC}^n$ by the equation 
$\pf = 0$. Then we can write $D = \varphi^{-1}(L)$. By this description, we see that the standard 
Poisson structure of $\Lieso(V)$ restricts to give a Poisson structure on $D$. As an open set, $\mathrm{Hom}'(V,U)/\LieSp(U)$ also inherits a Poisson structure. Since $\rho$ is the moment map 
for the $\LieSO(V)$-action, this Poisson structure coincides with the natural Poisson structure 
induced by the quotient map $\mathrm{Hom}'(V,U) \to \mathrm{Hom}'(V,U)/\LieSp(U)$. This implies that 
the diagram 
$$D \stackrel{\rho'}\longleftarrow \mathrm{Hom}'(V,U) \stackrel{\pi'}\longrightarrow \Liesp(U)$$
is a full dual pair. The symplectic leaf of $\Liesp(U)$ passing through $\pi'(X_0)$ is the nilpotent 
orbit $O_{[2n-i-1,i-1]}$. Similarly, the symplectic leaf of $D$ passing through $\rho'(X_0)$ is the 
nilpotent orbit $O_{[2n-i,i]}$. We then see that  the transverse Poisson structure on $D$ at $\rho'(X_0)$ 
and the transverse Poisson structure on $\Liesp(U)$ at $\pi'(X_0)$ are anti-isomorphic by 
\cite[ Theorem 8.1]{W}.  By the $\LieSp(U)$-action and the $\LieSO(V)$-action, we may assume that 
$x_1 = \rho'(X_0)$ and $x_2 = \pi'(X_0)$. Therefore, there is an anti-isomorphism 
$(S_1\cap D,x_1)\cong (S_2, x_2)$ of Poisson structures.   

Let $(L,0)$ be the germ of $L$ at the origin. Restricting the adjoint quotient map $\varphi$ to 
$S_1 \cap D$,  we get a map $(S_1 \cap D, x_1) \to (L,0)$. On the other hand, let 
$\varphi':\Liesp(U) \to {\IC}^{n-1}$ be the adjoint quotient map for $\Liesp(U)$. Restricting 
$\varphi'$ to $S_2$, we get a map $(S_2, x_2) \to ({\IC}^{n-1},0)$. Since the isomorphism between 
$(S_1 \cap D, x_1)$ and $(S_2, x_2)$ preserves symplectic leaves, we finally have a commutative 
diagram 
\begin{equation}\label{eq:TwoSlices2CD} 
\begin{CD} 
(S_1 \cap D, x_1) @>>> (S_2, x_2) \\ 
@VVV @VVV \\ 
(L,0) @>>> ({\IC}^{n-1},0).     
\end{CD} 
\end{equation}
\end{proof}


\end{document}